\documentclass[times]{oupau}

\usepackage{fancybox,color,graphicx,amsfonts,amsmath,amssymb,color}

\usepackage[all,color,line]{xy}
\newtheorem{knowntheorem}{Theorem}

\newtheorem{claim}{Claim}
\theoremstyle{remark}
\newtheorem{remark}[theorem]{Remark}
\numberwithin{equation}{section}
\newcommand{\mx}{\mathbf x}
\newcommand{\mX}{\mathbf X}

\newcommand{\mW}{\mathbf W}

\newcommand{\la}{\lambda}
\newcommand{\tr}{{\rm tr}}\newcommand{\eps}{\varepsilon}

\newcommand{\RR}{\mathbb{R}}
\newcommand{\CC}{\mathbb{C}}
\newcommand{\NN}{\mathbb{N}}
\newcommand{\sC}{\mathcal{C}}
\newcommand{\Sn}{\mathcal{S}}
\newcommand{\WW}[1]{\mathcal{W}\left(#1\right)}
\newcommand{\WWq}[1]{\mathcal{W}_q\left(#1\right)}

\newcommand{\E}{\mathbb{E}}
\newcommand{\EE}{\tau}
\newcommand{\calF}{\mathcal{F}}
\newcommand{\calG}{\mathcal{G}}
\newcommand{\calH}{\mathcal{H}}
\newcommand{\calA}{\mathcal{A}}

\newcommand{\calM}{\mathcal{M}}
\newcommand{\calV}{\mathcal{V}}

\newcommand{\XX}{{\mathbf X}}
\newcommand{\YY}{{\mathbf Y}}
\newcommand{\ZZ}{{\mathbf Z}}

\newcommand{\sR}{{\mathbb R}}

\newcommand{\CR}{{\rm cr}}

\newenvironment{proofof}[1]{\par\medbreak
 \textit{#1.}\hskip.5em\ignorespaces}{\medbreak}

\begin{document}

\title{Compound real Wishart and $q$-Wishart matrices}

\author{W{\l}odzimierz  Bryc\affil{a}}
\address{
\affilnum{a}Department of Mathematical Sciences, University of Cincinnati, 2855
Campus Way, PO Box 210025, Cincinnati, OH 45221-0025, USA. Email: Wlodzimierz.Bryc@UC.edu}

\date{Created:  April 8, 2007. Printed: \today}


\received{October 5, 2007}
\revised{March 13, 2008}
\begin{abstract}
We introduce a family of matrices with
non-commutative entries that generalize the
classical real Wishart matrices.
 With the help of the Brauer product, we derive a
 non-asymptotic
expression for the
 moments of traces of monomials in such matrices; the expression is quite similar to the formula
 derived in  \cite[Theorem 2.1]{Bryc-06} for independent complex Wishart
 matrices.

 We then analyze the fluctuations about the Marchenko-Pastur law. We  show that after centering by the mean,
  traces of real symmetric polynomials in $q$-Wishart matrices converge in distribution, and we identify the asymptotic law as
  the normal law when $q=1$, and as the semicircle law when $q=0$.
\end{abstract}
\tolerance=1000

\maketitle
\section{Introduction}
The real
Wishart distribution was introduced by Wishart \cite{Wishart-1928}. As evidenced by the vast literature, the Wishart law
 is of primary importance to statistics, see e.g. \cite{Anderson-03,Muirhead}.
The compound Wishart distribution was introduced by Speicher \cite{Speicher-98}.
We will consider a slightly
more general class of laws on random matrices which was introduced
 in Ref.
\cite{Hanlon-Stanley-Stembridge-92}, and also its non-commutative generalization. 

 \begin{definition} Let $\Sigma$ be an $N\times N$
 real positive-definite matrix (i.e. $x^*\Sigma x>0$ for all non-zero
 $x\in\RR^N$). Let $B$ be an $M\times M$ real
 matrix. We will say that a random matrix $\mW$ is Wishart with shape parameter $B$ and scale parameter $\Sigma$, if
\begin{equation}
  \label{Def gW}
  \mW=A'\mX' B \mX A,
\end{equation}
where $\mX$ is an $M\times N$ matrix of i.i.d. $N(0,1)$ random variables, and $A$ is a symmetric root,
 $A'A=A^2=\Sigma$. (Here and throughout the paper, $A'$ denotes the transpose of $A$.)
We write $\mW\in\WW{\Sigma,B}$.
  \end{definition}

When $B$ is positive-definite, Ref. \cite{burda-2006} interprets $\mW$ as
a sample covariance under correlated sampling.
Similar interpretation for the complex case appears in Ref. \cite{Chuah-Tse-Kahn-Valenzuela-02}.
Somewhat more generally, one could consider
all matrices of the form \eqref{Def gW} with non-symmetric $A$; the  moments of such matrices would then depend on
$\Sigma^+=A'A$ and $\Sigma^-=AA'$ as well as on $B,B'$.

The usual real Wishart matrices correspond to  the choice $B=I_M$, so
the notation $W(\Sigma,M)$  in \cite[Section 7.2]{Anderson-03} means $\WW{\Sigma,I_M}$ in our notation.
Ref.  \cite{Graczyk-Letac-Massam-05} uses a slightly different parametrization  $\sigma=2\Sigma$, $p=M/2$,
which should be kept in mind when comparing the formulas.
Compound Wishart matrices are the elements of $\WW{I_N,B}$, see \cite[Section 4.4]{Speicher-98} or
\cite[page 169]{Hiai-Petz}.

Our main result about real Wishart matrices, Theorem \ref{T1},  gives a
 non-asymptotic
expression for the
 moments of traces of monomials in several independent real Wishart matrices. This formula is similar to the expression
 derived in  \cite[Theorem 2.1]{Bryc-06} for independent complex Wishart
 matrices. The statement requires additional notation on Brauer products which will be introduced in Section \ref{Sect 1}.
  Our main application of this formula deals with the fluctuations about the Marchenko-Pastur law.
The latter problem fits naturally into a more general non-commutative setting, see Theorem  \ref{T4}.

We now introduce 
matrices which arise from matrices with  noncommutative $q$-Gaussian entries by
a formula similar to \eqref{Def gW}.
The noncommutative $q$-Gaussian random variables were introduced on the  algebraic level
by Frisch and Bourret \cite{FB70}; additional physical applications are in
\cite{Parisi-94,vanLeeuwen-Maassen-98}. The Fock-space construction is due to Bo\.zejko and Speicher
\cite{Bozejko-Speicher-91}; an excellent exposition of this topic
appears in Ref. \cite{BKS-97}.

Fix a parameter $-1\leq q\leq 1$. Let $\calA$  be a noncommutative probability space, i.e., a
 unital complex $*$-algebra with a tracial state $\EE:\calA\to\CC$; see, e.~g.,
\cite[Sect. 1.2]{Hiai-Petz} or \cite{VDN92} for more details and examples.
\begin{definition}[{\cite[Theorem 1]{FB70}}]\label{q-Prob Space}
 $\XX_1,\XX_2,\dots\in\calA$ are $q$-Gaussian, if  $\XX_j=\XX_j^*$, and for any finite collection of
 $\YY_1,\dots,\YY_n\in\{\XX_1,\XX_2,\dots\}$,
\begin{equation*}
  \label{q-Wick}
  \EE(\YY_{1}\dots \YY_{n})=
    \sum_{\calV}
  \displaystyle q^{\CR(\calV)}\prod_{\{i,j\}\in\calV} \EE( \YY_i \YY_j).
\end{equation*}
The sum is over all pair partitions $\calV$ of  $\{1,\dots,n=2k\}$
 and to be interpreted as $0$ when $n$ is odd;
$cr(\calV):=\#\{(i_1,i_2): \{i_1,j_1\}, \{i_2,j_2\}\in\calV, i_1<i_2<j_1<j_2\}$ is the number
of crossings of a pair partition $\calV$.
\end{definition}

Consider now the $*$-probability space $\calM_{N\times N}(\calA)$ of
 all $N\times N$ matrices with elements from $\calA$,
with matrix multiplication as the product operation, with state
$$\EE\circ\tr_N\left( \XX\right):=\frac1N\sum_{i=1}^N\EE([\XX]_{i,i}),$$
and with the natural definition of the star operation as conjugation of elements in $\calA$, and transposition.

For $N\geq 1$, a positive definite real  $N\times N$ matrix $\Sigma$ and an $M\times M$
positive definite real matrix $B$, let $\mX_{i,j}$ be $q$-Gaussian noncommutative random variables such that
 for $1\leq i,k\leq M$, $1\leq j,m\leq N$,
\begin{equation}
  \label{q-factorization}
  \EE(\XX_{i,j}\XX_{k,m})=[B]_{i,k}[\Sigma]_{j,m}.
\end{equation}

\begin{definition}
  \label{Def q-Wishart}
We will say that  $\mW\in\calM_{N\times N}(\calA)$ is a $q$-Wishart random variable
   with shape parameter $B$ and scale parameter $\Sigma$ if
  \begin{equation}\label{q-Wigner}
    \mW=\mX^* \mX,
  \end{equation}
  where
$\mX=\left[ \XX_{i,j}\right]$
is the $M\times N$ matrix with $q$-Gaussian entries that satisfy \eqref{q-factorization}.
 We will write $\mW\in\WWq{\Sigma,B}$.
\end{definition}
(Matrices with noncommutative entries have been studied in
\cite{Nica-Shlyakhtenko-Speicher-00,Ryan-CMP-98,Shlyakhtenko-97,Thorbjornsen-00}.
According to \cite[Definition 1.2.2]{Mingo-Nica-01}
when $\Sigma=I_N$, $B=I_M$, $\XX_{i,j}$ form a $q$-circular system.)

We will also need a substitute for independence.
\begin{definition}
  We will say that $\mW_1,\mW_2,\dots,\mW_s\in\WWq{\Sigma,B}$ are $q$-orthogonal
  if the entries of the corresponding matrices \eqref{q-Wigner}:
  $\{\XX_{i,j}^{(1)}\}$ for $\mW_1$,  $\{\XX_{i,j}^{(2)}\}$ for $\mW_2$, ..., are pairwise
    uncorrelated, $\EE(\XX_{i,j}^{(k)}\XX_{i',j'}^{(k')})=[B]_{i,i'}[\Sigma]_{j,j'}\delta_{k,k'}$.
\end{definition}

\begin{remark}
When $q=1$, all $q$-Gaussian random variables commute and their joint moments are the moments of the multivariate normal
 law of matching covariance.
In particular, for every finite $q$-Gaussian family $\XX_1,\dots,\XX_m$, there exists an $m\times m$ real matrix
$A$ and independent classical standard normal random variables $Z_1,\dots,Z_m$ such that  with $Y_j=\sum_k [A]_{jk}Z_k$,
  $$
  \EE( \XX_1^{k_1}\dots \XX_m^{k_m})=\E(Y_1^{k_1}\dots Y_m^{k_m}) \mbox{ for all $k_1,\dots,k_m\geq 0$}.
  $$
Thus the joint moments of $q$-orthogonal  $\mW_1,\mW_2,\dots,\mW_s\in\WWq{\Sigma,B}$ coincide with the respective joint moments
of independent real Wishart matrices. Under appropriate assumptions, such matrices are
asymptotically free and their limit laws are compound Marchenko Pastur laws. The conclusion of Theorem \ref{T4} in this case
means that for real Wishart matrices the fluctuations around the Marchenko-Pastur law are asymptotically normal.
(This fact is known for
complex Wishart matrices, see \cite{Mingo-Nica-04}.)
\end{remark}
\begin{remark}\label{R2}
  When $q=0$,  for every finite $q$-Gaussian family $\XX_1,\dots,\XX_m$, there exists an $m\times m$ real matrix
$A$ and free semicircular elements  $\ZZ_1,\dots,\ZZ_m$ such that   with $\YY_j=\sum_k [A]_{jk}\ZZ_k$,
  $$
  \EE( \XX_{j_1}\dots \XX_{j_n})=\EE( \YY_{j_1}\dots \YY_{j_n}) \mbox{ for all $1\leq j_1,\dots,j_n\leq m$}.
  $$
  Thus, for example, the traces of powers of  $q$-orthogonal  $\mW_1,\mW_2,\dots,\mW_s\in\WWq{\Sigma,B}$  are free in $(\calA,\EE)$.
In Proposition \ref{T3} we show that   $\mW\in\WWq{\frac1NI_N,B}$ as
an element of $(\calM_{N\times N}(\calA),\EE\circ\tr_N)$
  has a compound Marchenko-Pastur law already for finite $N$,
so on average the fluctuations around the limit are zero. Theorem \ref{T4} shows that the un-averaged
fluctuations around the Marchenko-Pastur law have a non-trivial semicircular limit.
\end{remark}

%

Asymptotic normality for  traces of polynomials in several independent complex Wishart
matrices is established in Ref. \cite{Mingo-Nica-04}. This case served as a departure point for a number of papers
on second order freeness, see e.g. \cite{Collins-Mingo-Siniady-Speicher-06} and the references therein.
 For an extension to more general covariances see also
\cite[Theorem 1]{Bryc-06}. Asymptotic normality for one real Wishart matrix is established in {\cite[Theorem 2]{Arharov-71}, \cite[Theorem 4.1]{Jonsson-82}}. Asymptotic normality of self-adjoint polynomials in noncommutative entries in a different setting appears in
  \cite[Theorem 1]{Kuperberg-05}. It turns out that in our noncommutative setting, asymptotic limit law always exists,
but it depends on the polynomial in a nontrivial way.

\begin{theorem}\label{T4}
Suppose $\mW^{(N)}_1,\mW^{(N)}_2,\dots,\mW_s^{(N)}\in\WWq{\frac1NI_N,I_M}$ are $q$-orthogonal and $M=M(N)$ is such that
$M/N\to\la\in[0,\infty)$ and $N\to\infty$.
  If $Q()$ is a fixed non-commutative polynomial in $s$ variables,  and
  $$X_{N}:=\tr(Q(\mW^{(N)}_1,\dots,\mW^{(N)}_s))-\EE\left(\tr(Q(\mW^{(N)}_1,\dots,\mW^{(N)}_s))\right),$$
then $X_{N}$ converges in moments as $N\to\infty$.
Furthermore, if $Q$ is a real symmetric polynomial so that $X_{N}$ is self-adjoint,
then  $X_{N}$ converges in distribution; if in addition  $q=1$ then $X_{N}$ has asymptotically normal law;
  if $q=0$ then $X_{N}$ has asymptotically semicircle law.
\end{theorem}

Theorem \ref{T4} raises an interesting question about how to identify or describe
the limiting law for  all $q\in[-1,1]$.
Our proof  relies on a general property of $q$-Wishart matrices, see
Theorem \ref{T5}, and freeness or independence when $q=0$ or $q=1$; this technique is not helpful in general.

Computer-assisted calculations based on exact formulas from Theorem \ref{T2} indicate that
the
limit law  for general $q$ depends on the polynomial in a complicated way, see Example \ref{Ex3}.
Similar situation was noted for  Wigner matrices in noncommutative entries in an
unpublished manuscript \cite{Speicher}.

For one matrix case, one may be able to find  polynomials with
limits that can still be described in simple terms.
One such polynomial seems to be $\tr(\mW^{(N)})$, see Example
\ref{T4-counterexample};
another suggestive computer-assisted calculation is in Example \ref{Ex 4}.

%

We now list two corollaries.
We begin with  traces of polynomials of a single $q$-Wishart matrix.
\begin{corollary}\label{CT4}
  If $Q$ is a real polynomial,  $\mW^{}\in\WWq{\frac1NI_N,I_M}$, $M/N\to \la\in[0,\infty)$,
   then
  $$X_{N}:=\tr(Q(\mW^{}))-\EE\left(\tr(Q(\mW^{}))\right)$$
  converges in distribution as $N\to\infty$.
  Furthermore, if $q=1$ then $X_{N}$ has asymptotically normal law;
  if $q=0$ then $X_{N}$ has asymptotically semicircle law.
\end{corollary}

For $q=1$, the traces commute, and their joint moments coincide with the moments of traces of the corresponding
 real Wishart matrices \eqref{Def gW}; taking
$Q(\mW^{}_1,\dots,\mW^{}_s)=\sum_{j=1}^k a_j \prod_{r=1}^j\mW_{t(r)}$,
  by
the Cramer-Wald device (\cite[Theorem 29.4]{Billingsley-95})
we deduce the following multimatrix generalization of {\cite[Theorem 2]{Arharov-71},
 \cite[Theorem 4.1]{Jonsson-82}} which
shows that \cite[Corollary 9.4]{Mingo-Nica-04} holds also for real Wishart matrices.
\begin{corollary}
  Fix $s\geq 1$ and $t:\NN\to \{1,\dots,s\}$. Suppose $\mW_1,\dots,\mW_s\in\WW{\frac1NI_N,I_M}$ are independent.
  For a fixed  $k\geq 1$ that does not depend on $N$, and for
  $1\leq j\leq k$ consider real
  random variables
 $$
X_{j}^{(N)}=\tr\left(\prod_{r=1}^j\mW_{t(r)}\right)-\E\left(\tr(\prod_{r=1}^j\mW_{t(r)})\right).
  $$
Then as $N\to\infty$, $M/N\to\la\in[0,\infty)$, the sequence of $k$-dimensional random variables
$$ \left(X_{j}^{(N)}\right)_{1\leq j\leq k}$$
converges in distribution to the multivariate normal law.
\end{corollary}

The paper is organized as follows.
In Section \ref{Sect 1} we introduce notation,
state our results on mixed moments of traces of monomials in independent Wishart matrices and on mixed moments of
 traces of monomials
 in $q$-orthogonal $q$-Wishart matrices,
  and we deduce
 some of the related expressions available in the literature.
Section \ref{Sect Proof 1} contains proofs of the formulas for moments.
Section \ref{Sect Proof of T4} contains some auxiliary results, and the proof of Theorem \ref{T4}.

%

\section{Moments of  real Wishart and $q$-Wishart matrices}\label{Sect 1}

Explicit formulas for moments of polynomials in one real Wishart matrix appear in
\cite{Graczyk-Letac-Massam-05,Graczyk-Vostrikova-06,Hanlon-Stanley-Stembridge-92,Lu-Richards-01}.
The formulas are more complicated than the formulas for the
complex case \cite{Graczyk-Letac-Massam-03,Hanlon-Stanley-Stembridge-92,Mingo-Nica-04}, and often
involve sophisticated tools, like wreath products, Jack polynomials, and properties of hyperoctahedral group.

One of the main results of this paper is the
 non-asymptotic
expression for
 moments of traces of monomials in several independent real Wishart matrices, and its noncommutative
 $q$-generalization.  In Section \ref{Sect Proof of T4} we use the moment formula to prove Theorem \ref{T4}, but  since the formula is of independent interest  we consider a more general case than what we need for the proof.
Our approach follows closely \cite[Theorem 2]{Bryc-06}: it relies on Wick formula,
combinatorics from \cite{Mingo-Nica-04} and the Brauer product \cite{Brauer-1937} of pair partitions.
The connection of the Brauer algebra to integration over the orthogonal group  has already been noted in
\cite{Collins-Sniady-06} who seem not to use
the Brauer
product explicitly.
With the use of the Brauer product,
 our formula closely resembles the formula for the complex
 case.

Following the referee's comment we remark that we use
Brauer algebra as a convenient notational replacement for the permutation group, but since we
multiply only by elements from $\calF_n^+$, our use  of the Brauer
 structure is in fact the action of $S_n$ on $2n$-pair partitions, as in \cite{Graczyk-Letac-Massam-05}.
In particular, the constant of the Brauer algebra does not enter our formulas.


After the first version of this paper was written,
we learned from J. Mingo about his
work \cite{mingo-2008} with Emily Redelmeier who has found a very elegant
diagrammatic interpretation of the right hand side of
equation \eqref{Main formula}.

\subsection{Pair partitions, permutations, Brauer product}
Consider matrix variables $\mx_1,\mx_2,\dots,\mx_s$ which we interpret as variables of ``colors"
$1, \dots,s$.
We are interested in  monomials $p_{\gamma,t}(\mx_1,\mx_2,\dots,\mx_s)$
of degree $n$ which we parameterize   by
perfect matchings (pair partitions) $\gamma$ of $\{\pm1,\pm2,\dots,\pm n\}$,  and
by functions $t:\{1,\dots,n\}\to\{1,\dots,s\}$, which we interpret as assigning ``colors" to the integers.
To introduce these polynomials in formula \eqref{r} below we need
 additional notation and terminology.
\subsubsection{Pair partitions}
Let $\calF_n$ denote the set of pair partitions (i.e., partitions into two-element sets)
of the $2n$-element set $\{\pm 1,\pm 2,\dots,\pm n\}$.  Ref. \cite{Goulden-Jackson-96} uses
the notation $\{\hat{1},1,\hat{2},2,\dots,\hat{n},n\}$ for the same object - the ordered
 set of $2n$ elements with
the distinguished match.
Other authors interpret $\calF_n$ as the set of
 1-factors, or perfect matchings, or 1-regular graphs on the vertices $\{\pm 1,\pm 2,\dots,\pm n\}$.

Identifying the pairs $\{i,j\}$ of a pair partition $\gamma\in\calF_n$ with the cycles of the permutation,
we can embed $\calF_n\subset \Sn_{\{\pm 1,\dots,\pm n\}}\cong\Sn_{2n}$.
Here $\Sn_F$ denotes the group of permutations (i.e., bijections with composition) of a finite set F.
As usual $\Sn_n=\Sn_{\{1,\dots,,n\}}$. In the embedding, $\calF_n$ is the set of all involutions on $\{\pm 1,\dots,\pm n\}$
with no fixed points.
Accordingly,  we write $\gamma(i)=j$ if $\{i,j\}\in\gamma$.\

We will also find it convenient to represent $\gamma$ as a graph with vertices arranged in
two rows
$$\begin{array}{rrrr}
  1 &2 &\dots & n \\
  -1 & -2 &\dots& -n
\end{array}$$
with the edges drawn between the vertices in each pair of  partition $\gamma$.
For example,  a pair partition $\delta$ which pairs $j$ with $-j$ for $j=1,\dots,n$
is  identified with the permutation $\delta(j)=-j$, or with the graph

\begin{equation}
  \label{eq:delta}
  \delta=\begin{matrix}
  \xymatrix  @-1pc{
 {^1_ \bullet} \ar@{-}[d]& {^2_\bullet}
  \ar@{-}[d]& \dots & {^n_\bullet} \ar@{-}[d]\\
 {^{\hspace{1.5mm}\bullet}_{-1}} &{^{\hspace{1.5mm}\bullet}_{-2}} &\dots &
{^{\hspace{1.5mm}\bullet}_{-n}} \\
}
\end{matrix}
\end{equation}
In fact, $\delta$ is the distinguished matching on $\{\pm 1,\dots,\pm n\}$ and
 plays a special role in several definitions that follow.

With each $\gamma\in\calF_n$ we associate
the permutation $\pi(\gamma)\in\Sn_n$ and the sequence $\eps(\gamma)$ of $\pm$.
These mappings $\pi:\calF_n\to\Sn_n$ and $\eps:\calF_n\to \{-,+\}^n$ are defined in terms of graphs as follows.

Consider the 2-regular graph $\delta\cup \gamma$.
To define the first cycle of permutation $\alpha:=\pi(\gamma)$,
 start with vertex $-1$. Follow upwards the first edge of $\delta$ towards $1$, and then continue
 by following
 the consecutive edges of the graph $\delta\cup \gamma$.
The first cycle of
$\alpha$ consists of the vertices from  the upper row $\{1,2,\dots,n\}$ that were
 visited by this path, taken in the order in which they were visited.
The corresponding values
of $\eps(\gamma)$ on this cycle are  determined by the direction of traversing the edges of $\delta$:
we assign $+$ if the corresponding edge of $\delta$ was traversed upwards, and $-$ otherwise.
(In particular, the sequence $\eps(\gamma)$ of signs always starts with $+$.)

The next cycle of $\pi(\gamma)$ consists of the entries in the first row, starting with the leftmost element $j$ of the first row that was has not yet been
visited. To determine the cycle, we again start at the corresponding vertex $-j$ from the lower row, follow upwards the edge of $\delta$ towards
 $j$, and continue in the same direction along the edges of $\delta\cup \gamma$ until we complete the second cycle.
We then repeat these steps until we exhaust all the elements of the first row.
(Similar construction appears in Refs.
\cite{Goulden-Jackson-96,Graczyk-Letac-Massam-05,Hanlon-Stanley-Stembridge-92}.)

For example, suppose
$$
\gamma=\{\{1,2\}, \{3, -4\}, \{-1,-2\},\{-3,4\} \}=\begin{matrix}
  \xymatrix  @-1pc{
 \bullet \ar@{-}[r]&   \bullet & \bullet\ar@{-}[dr] &\bullet\ar@{-}[dl]\\
 \bullet \ar@{-}[r]&\bullet &\bullet &\bullet
}
\end{matrix}
$$
Then
$$
\delta\cup\gamma=\begin{matrix}
  \xymatrix  @-1pc{
 \bullet \ar@{-}[r] \ar@{<-}[d]&   \bullet & \bullet\ar@{-}[dr] &\bullet\ar@{-}[dl]\\
 \bullet \ar@{-}[r]&\bullet\ar@{-}[u]  &\bullet\ar@{->}[u]  &\bullet\ar@{-}[u]
}
\end{matrix}
$$
Here  the edges for the initial vertices of each cycle are oriented in the direction of the "first step".
Thus $\pi(\gamma)=(1,2)(3,4)$ and $\eps(\gamma)=(+,-,+,+)$.

Next, we introduce partitions that will play a role in Section \ref{Sect NC}. 
\begin{example} For $\la\vdash n$, i.e. for $\la=(\la_1\geq \la_2\geq\dots \geq 0)$ with $\sum \la_j=n$,
 consider
\begin{equation}\label{sigma_la}
\sigma_\la=\left\{\begin{array}[t]{c}\underbrace{\{1,-2\},\{2,-3\},\dots,\{\la_1,-1\}}\\ \la_1\mbox{ pairs}\end{array},
\begin{array}[t]{c}\underbrace{\{\la_1+1,-\la_1-2\},\dots,
\{\la_1+\la_2,-\la_1-1\}}\\ \la_2 \mbox{ pairs}\end{array},\dots\right\}.
\end{equation}
Then
$$
\sigma_\la=\begin{matrix}
 \xymatrix  @-1pc{
 {^1_ \bullet} \ar@{-}[dr]& {^2_\bullet} \ar@{-}[dr]& {^3_\bullet}
 & \dots & {^{\la_1}_\bullet}\ar@{-}[dllll] & {^{\la_1+1}_{\hspace{2.5mm}\bullet}}\ar@{-}[dr] & {^{\la_1+2}_{\hspace{3mm}\bullet}}&\dots & {^{\la_1+\la_2}_{\hspace{3.7mm}\bullet}}\ar@{-}[dlll]&
  {^{\la_1+\la_2+1}_{\hspace{5mm}\bullet}}&\dots \\
 {^{\hspace{1.5mm}\bullet}_{-1}} &{^{\hspace{1.5mm}\bullet}_{-2}} &{^{\hspace{1.5mm}\bullet}_{-3}}&\dots &
{^{}_\bullet}& {^{}_\bullet} & {^{}_\bullet}&\dots&{^{}_\bullet}&{^{}_\bullet}&\dots
}
\end{matrix}
$$
and
$$
\sigma_\la\cup\delta=\begin{matrix}
 \xymatrix  @-1pc{
 {^1_ \bullet} \ar@{-}[dr]& {^2_\bullet} \ar@{-}[dr]& {^3_\bullet}
 & \dots & {^{\la_1}_\bullet}\ar@{-}[dllll] & {^{\la_1+1}_{\hspace{2.5mm}\bullet}}\ar@{-}[dr] & {^{\la_1+2}_{\hspace{3mm}\bullet}}&\dots & {^{\la_1+\la_2}_{\hspace{3.7mm}\bullet}}\ar@{-}[dlll]& {^{\la_1+\la_2+1}_{\hspace{5mm}\bullet}}&\dots \\
 {^{\hspace{1.5mm}\bullet}_{-1}}\ar@{->}[u] &{^{\hspace{1.5mm}\bullet}_{-2}}\ar@{-}[u] &{^{\hspace{1.5mm}\bullet}_{-3}}\ar@{-}[u]&\dots &
{^{}_\bullet}\ar@{-}[u]& {^{}_\bullet} \ar@{->}[u]& {^{}_\bullet}\ar@{-}[u]&\dots&{^{}_\bullet}\ar@{-}[u]& {^{}_\bullet} \ar@{->}[u]&\dots
}
\end{matrix}
$$
Thus
$\pi(\sigma_\la)=(1,2,\dots,\la_1)(\la_1+1,\dots,\la_1+\la_2),\dots$
has cycle type $\la$.
In particular, $\delta=\sigma_{1^n}$ and $\pi(\delta)=id$. Similarly, $\pi(\sigma_n)=(1,2,\dots,n)$.
\end{example}

\begin{definition}
  By $\calF_n^+$ we denote the set of pair partitions $\sigma\in\calF_n$ in which the vertices of
the upper row are always connected to vertices of the lower row.
Such pair partitions are characterized by the fact that
$\eps=(+,+,\dots,+)$.
\end{definition}
\begin{remark}\label{Remark: bijection} Brauer \cite{Brauer-1937} identifies the elements of $\calF_n^+$
with permutations,
as $\pi:\calF_n^+\to \Sn_n$ is a bijection. For $\sigma\in\calF_n^+$, the definition of
$\pi(\sigma)$ simplifies to
\begin{equation}
  \label{pi on F+}
  \pi(\sigma)(j):=-\sigma(j), \;j=1,2,\dots,n.
\end{equation}
The inverse mapping
$\Sn_n\ni \alpha\mapsto \sigma\in\calF_n^+$ is then
$
\sigma(j)=-\alpha(j)$.
\end{remark}
\subsubsection{Brauer product of pair partitions}
The Brauer product of pair partitions has simple description in terms of graphs.
To define $\sigma\circledcirc \gamma$ we draw the two rows for $\gamma$, then we draw the graph for $\sigma$
so that its upper row covers the lower row of $\gamma$. We construct the new two-row graph
for $\sigma \circledcirc \gamma$ by removing the middle row, but retaining the edges between the
vertices of the first and third row that were connected in the three-row graph.
Brauer \cite{Brauer-1937} weights the resulting graph by a coefficient
that depends on the number of cycles lost in the middle row; we will apply
the Brauer product only when there are no such cycles, so our weight will always be 1.

\begin{example}
Suppose
$$
\sigma=\{\{1,-2\}, \{2,-3\},\{3, -4\}, \{4,-1\} \}=\begin{matrix}
  \xymatrix  @-1pc{
 \bullet \ar@{-}[rd]&   \bullet \ar@{-}[rd]& \bullet\ar@{-}[dr] &\bullet\ar@{-}[dlll]\\
 \bullet &\bullet &\bullet &\bullet
}
\end{matrix}
$$
$$
\gamma=\{\{1,2\}, \{3, -4\}, \{-1,-2\},\{-3,4\} \}=\begin{matrix}
  \xymatrix  @-1pc{
 \bullet \ar@{-}[r]&   \bullet & \bullet\ar@{-}[dr] &\bullet\ar@{-}[dl]\\
 \bullet \ar@{-}[r]&\bullet &\bullet &\bullet
}
\end{matrix}
$$

The intermediate three-row graph is
$$
 
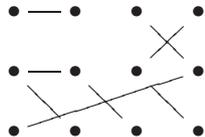
\xymatrix  @-1pc{
 \bullet \ar@{-}[r]&   \bullet & \bullet\ar@{-}[dr] &\bullet\ar@{-}[dl]\\
 \bullet \ar@{-}[r]\ar@{-}[rd]&\bullet\ar@{-}[rd] &\bullet\ar@{-}[rd] &\bullet\ar@{-}[dlll]\\
 \bullet &\bullet &\bullet &\bullet
}
$$
This gives
$$
\sigma \circledcirc \gamma=\begin{matrix}
  \xymatrix  @-1pc{
 \bullet \ar@{-}[r]&   \bullet & \bullet\ar@{-}[dll] &\bullet\ar@{-}[d]\\
 \bullet &\bullet \ar@{-}[r]&\bullet &\bullet
}
\end{matrix}=\{\{1,2\},\{-2,-3\},\{-1,3\},\{-4,4\}\}.
$$
\end{example}

 Our formulas
  need only the Brauer products with $\sigma\in\calF_n^+$, and in this case the analytical definition takes the following form.
\begin{definition}\label{Def B-comp}  For $\sigma\in\calF_n^+$
and $\gamma\in\calF_n$, the Brauer product \cite{Brauer-1937}
$\sigma \circledcirc \gamma$ is defined as an element of $\calF_n$ such that
\begin{equation}
  \label{B-conv}
  (\sigma \circledcirc \gamma)(j)=\begin{cases}
  \sigma(-\gamma(j)) & \mbox{if } j>0,\; \gamma(j)<0, \\
 \gamma(j) &  \mbox{if } j>0,\; \gamma(j)>0, \\
     \sigma(-\gamma(-\sigma(j)))&  \mbox{if } j<0,\;  \gamma(-\sigma(j))<0,\\
        \gamma(-\sigma(j)) &  \mbox{if }   j<0,\; \gamma(-\sigma(j))>0.
\end{cases}
\end{equation}
\end{definition}
(These are the same cases that are listed in Table \ref{Tbl2}.)

\begin{remark}[Brauer] If $\sigma_1,\sigma_2\in\calF_n^+$ then $\sigma_1\circledcirc\sigma_2\in\calF_n^+$
and
 \begin{equation}\label{pi-composition}
\pi(\sigma_1\circledcirc\sigma_2)=\pi(\sigma_1)\circ\pi(\sigma_2)\end{equation}
 becomes just the ordinary composition of permutations.
\end{remark}

We also remark that $\delta\circledcirc\gamma=\gamma$ and, when the more general definition is used,
$\gamma\circledcirc\delta=\gamma$.


It will be convenient to have concise notation for certain objects and properties
of $\gamma\in\calF_n$ which depend only through the associated permutation $\pi(\gamma)\in\Sn_n$.
\begin{definition}
  \label{Def cycle}
  For a ``coloring" $t:\{1,\dots,n\}\to\{1,\dots s\}$, by
  $$\calF_n(t)=\{\gamma\in\calF_n: t(|\gamma(j)|)=t(|j|), j=\pm 1,\dots,\pm n\}$$
  we denote the set of color-preserving pair partitions. Such partitions are characterized by the property that
  the permutation $\alpha:=\pi(\gamma)$ is color-preserving,
  $t\circ\alpha= t$, see \cite[(4)]{Bryc-06}.
\end{definition}
By $\sC(\gamma)$ we denote the set of cycles of permutation $\pi(\gamma)$.
For $\gamma\in\calF_n(t)$, all cycles have the same color;
by $\sC_j(\gamma)$ we denote the set of cycles of $\pi(\gamma)$ which are of color $j\in\{1,\dots,s\}$.

For a $\gamma\in\calF_n(t)$, we can color each edge of $\gamma$ by the color of
the adjacent vertices.
Since for $\sigma\in\calF_n^+$ each edge of
$\sigma\circledcirc\gamma$ arises from combining exactly one edge of $\gamma$ with some of the edges of $\sigma$,
we can color each edge of $\sigma\circledcirc\gamma$ by the color inherited from this
 unique edge of $\gamma\in\calF_n(t)$.

\begin{definition}\label{Def t-modified}
For $\gamma\in\calF_n(t)$, $\sigma\in\calF_n^+$,
 by $t(\sigma,\gamma)$ we denote the coloring of $\{1,2,\dots,n\}$ obtained as follows.
 For $1\leq i\leq n$, we define
 $t(\sigma,\gamma)(i)$ as the color of the edge of $\sigma\circledcirc\gamma$ which is crossed on
 $\delta\cup (\sigma\circledcirc\gamma)$ while
 connecting $i$ to the next element of the corresponding cycle of $\pi(\sigma\circledcirc\gamma)$.
 \end{definition}
\subsubsection{Formula for moments}

For a ``coloring" $t:\{1,\dots,n\}\to\{1,\dots s\}$, a pair partition $\gamma\in\calF_n$, and
matrix-variables $\mx_1,\mx_2,\dots,\mx_s$,
define monomial
\begin{equation}
  \label{r}
p_{\gamma,t}(\mx_1,\mx_2,\dots,\mx_s)=\prod_{c\in \sC(\gamma)}\tr\left(\prod_{j\in c} \mx_{t(j)}\right).
\end{equation}
Here the notation $\tr\left(\prod_{j\in c} \mx_{t(j)}\right)$ is to be interpreted as
the trace of the product of the matrices taken in the same order in which consecutive integers
appear in the cycle $c$.

For example, 
 $\tr^2(\mx_1\mx_2)=p_{\sigma,t}(\mx_1,\mx_2)$ with
 \begin{equation}
   \label{Ex0}
   \sigma=\begin{matrix}\xymatrix{
   \bullet \ar@{-}[dr] & \bullet \ar@{-}[dl] & \bullet \ar@{-}[dr] & \bullet \ar@{-}[dl]
    \\
    \bullet &\bullet &\bullet &\bullet }\end{matrix},\;  t(j)=(3+(-1)^j)/2,\; j=1,2,3,4.
 \end{equation}
Clearly, such representation is not unique, as we can realize
 the same polynomial $p_{\sigma,t}$ by combining an appropriate mapping $t$ with
 any element $\gamma\in\calF_n$
 such that $\pi(\sigma)$ and
$\pi(\gamma)$ are in the same conjugacy class.
The generic construction is to start with a partition $\la\vdash n$,
choose  $\sigma_\la\in \calF_n^+$ according to \eqref{sigma_la}, and then define a
suitable function
$t$.
This generic construction indicates that the monomials $p_{\gamma,t}$
are essentially the monomials that appear in \cite[Theorem 2]{Bryc-06},
with $h\equiv I$.

Identifying $\mx^{+}$ with $\mx$ and $\mx^{-}$ with the transpose $\mx'$ of a square matrix $\mx$, define
\begin{equation}
  \label{qr}
q_{\gamma,t}(\mx_1,\mx_2,\dots,\mx_s)=\prod_{c\in \sC(\gamma)}\tr\left(\prod_{j\in c} \mx_{t(j)}^{-\eps_j(\gamma)}\right).
\end{equation}
Of course, if matrix variables
$\mx_1,\dots,\mx_s$ are symmetric, then
$q_{\gamma,t}(\mx_1,\mx_2,\dots,\mx_s)=p_{\gamma,t}(\mx_1,\mx_2,\dots,\mx_s)$.
Trivially, for general matrix variables the polynomials coincide when $\gamma\in\calF_n^+$.

With the above notation, we  state the main result of this section.
\begin{theorem}
  \label{T1} Let $B_1,\dots,B_s\in\calM_{M\times M}$ and let
  $\Sigma_1,\dots,\Sigma_s$ be $N\times N$ positive definite matrices.
  If $\mW_1\in\WW{\Sigma_1,B_1},\mW_2\in\WW{\Sigma_2,B_2},\dots,\mW_s\in\WW{\Sigma_s,B_s}$
  are independent  and $\sigma\in\calF_n^+$, then
\begin{equation}
  \label{Main formula}
\E\left(p_{\sigma,t}(\mW_1,\mW_2,\dots,\mW_s)\right)=\sum_{\gamma\in \calF_n(t)}
q_{\gamma,t}(B_1,\dots,B_s)p_{\sigma\circledcirc\gamma,t(\sigma,\gamma)}(\Sigma_1,\Sigma_2,\dots,\Sigma_s).
\end{equation}
\end{theorem}
We postpone the proof to Section \ref{Sect Proof 1}.
In Section 
\ref{Sect NC} we give a noncommutative generalization.
This section concludes with  examples, remarks and corollaries.

\subsubsection{Example of mean and variance calculation}
The purpose of this section is to illustrate \eqref{Main formula} and also to show how the coloring
 $t(\sigma,\gamma)$ works  in  a simple case.

\begin{example}\label{Ex1} Under the assumptions of Theorem \ref{T1},
 formula \eqref{Main formula}   gives
$$\E\left(\tr(\mW_1\mW_2\dots\mW_s)\right)=\tr(B_1)\tr(B_2)\dots\tr(B_s)\tr(\Sigma_1\Sigma_2\dots\Sigma_s).$$
Indeed,  with $\sigma=\{(1,-2),(2,-3),(s-1,-s),(s,-1)\}$ and
$t(j)=j$, we get $\calF_s(t)=\{\delta\}$. Since $\sigma\circledcirc\delta=\sigma$ and $t(\sigma,\delta)=t$,
 we get
$$p_{\sigma\circledcirc\gamma,t(\sigma,\gamma)}(\mx_1,\dots,\mx_s)=
p_{\sigma,t}(\mx_1,\dots,\mx_s)=\tr(\mx_1\dots\mx_s),$$
and
$$q_{\delta,t}(\mx_1,\dots,\mx_s)=\tr(\mx_1')\dots\tr(\mx_s')=\tr(\mx_1)\dots\tr(\mx_s).$$
\end{example}
\begin{example}
  Suppose $\mW_1,\mW_2$ are as in Theorem \ref{T1}. From Example \ref{Ex1} we know that
 $\E(\tr(\mW_1\mW_2))=\tr(B_1)\tr(B_2)\tr(\Sigma_1\Sigma_2)$. To compute the variance of $\tr(\mW_1\mW_2)$ we therefore
 compute $\E(\tr^2(\mW_1\mW_2))$ using \eqref{Main formula} for $n=4$ with $\sigma,t$ as defined in \eqref{Ex0}.
{\footnotesize
\begin{table}[h]
\caption{\label{Table1}Terms in \eqref{Main formula} for calculating
 $\E(\tr^{2}(\mW_1\mW_{\color{blue}2}))$; here $\sigma$ is given by \eqref{Ex0}.
The second column  is an illustration for Section \ref{Sect NC}.}
\begin{tabular}{||c|c|c|c|c|c||}\hline\hline
$\gamma$ &$\CR(\gamma)$ &$\sigma\circledcirc\gamma$ &$\pi(\gamma)$ & $\pi(\sigma\circledcirc\gamma)$ &contribution \\ \hline\hline
$\begin{array}{c}\xymatrix @-1pc{
   \bullet \ar@{-}[d] & \bullet \ar@{-}[d] & \bullet \ar@{-}[d] & \bullet \ar@{-}[d]
    \\
    \bullet  &\bullet &\bullet &\bullet }\end{array}$
 & 0   & $ \begin{array}{c}\xymatrix @-1pc{
   \bullet\ar@{-}[dr]  & \bullet\ar@{-}[dl]  & \bullet\ar@{-}[dr]   & \bullet\ar@{-}[dl]
    \\
    \bullet &\bullet &\bullet &\bullet }\end{array}$&$(1)({\color{blue}2})(3)({\color{blue}4})$&
    $(1,{\color{blue}2})(3,{\color{blue}4})$
& $\tr^2(B_1)\tr^2(B_2)\tr^2(\Sigma_1\Sigma_2)$
    \\
    \hline 
$ \begin{array}{c}\xymatrix @-1pc{
   \bullet \ar@{-}[d] & \bullet \ar@/^/@{-}[rr]  & \bullet \ar@{-}[d] & \bullet
    \\
    \bullet  &\bullet \ar@/^/@{-}[rr] &\bullet &\bullet }\end{array}$
    &1& $ \begin{array}{c}\xymatrix @-1pc{
   \bullet\ar@{-}[dr]  & \bullet\ar@/^/@{-}[rr]  & \bullet\ar@{-}[dr]   & \bullet
    \\
    \bullet\ar@/^/@{-}[rr] &\bullet &\bullet &\bullet }\end{array}$
    &$(1)({\color{blue}2},{\color{blue}4})(3)$&$(1,{\color{blue}2},{4},{\color{blue}3})$ &
    $\tr^2(B_1)\tr(B_2B_2')\tr((\Sigma_1\Sigma_2)^2)$\\
   \hline  
$ \begin{array}{c}\xymatrix @-1pc{
   \bullet\ar@{-}[d]  & \bullet\ar@{-}[drr]  & \bullet \ar@{-}[d] & \bullet \ar@{-}[dll]
    \\
    \bullet  &\bullet &\bullet &\bullet }\end{array}$
    &0& $ \begin{array}{c}\xymatrix @-1pc{
   \bullet\ar@{-}[dr]  & \bullet\ar@{-}[dr]  & \bullet\ar@{-}[dr]   & \bullet\ar@{-}[dlll]
    \\
    \bullet &\bullet &\bullet &\bullet }\end{array}$&$(1)({\color{blue}2},{\color{blue}4})(3)$&
    $(1,{\color{blue}2},{3},{\color{blue}4})$
    &$\tr^2(B_1)\tr(B_2^2)\tr((\Sigma_1\Sigma_2)^2)$\\
 \hline   
$ \begin{array}{c}\xymatrix @-1pc{
   \bullet \ar@/^/@{-}[rr] & \bullet\ar@{-}[d]  & \bullet  & \bullet \ar@{-}[d]
    \\
  \bullet \ar@/^/@{-}[rr]  &\bullet &\bullet &\bullet }\end{array}$
    &1& $ \begin{array}{c}\xymatrix @-1pc{
   \bullet\ar@/^/@{-}[rr]  & \bullet\ar@{-}[dl]  & \bullet   & \bullet\ar@{-}[dl]
    \\
    \bullet &\bullet\ar@/^/@{-}[rr] &\bullet &\bullet }\end{array}$&$(1,3)({\color{blue}2})({\color{blue}4})$&
    $(1,{\color{blue}3},{4},{\color{blue}2})$& $\tr(B_1B_1')\tr^2(B_2)\tr((\Sigma_1\Sigma_2)^2)$\\
    \hline 
$ \begin{array}{c}\xymatrix @-1pc{
  \bullet \ar@/^/@{-}[rr] & \bullet \ar@/^/@{-}[rr]  & \bullet  & \bullet
    \\
  \bullet \ar@/^/@{-}[rr]  &\bullet \ar@/^/@{-}[rr] &\bullet &\bullet }\end{array}$
    &6& $ \begin{array}{c}\xymatrix @-1pc{
   \bullet\ar@/^/@{-}[rr]  & \bullet\ar@/^/@{-}[rr]  & \bullet   & \bullet
    \\
    \bullet\ar@/^/@{-}[rr] &\bullet\ar@/^/@{-}[rr] &\bullet &\bullet }\end{array}$&
    $(1,3)({\color{blue}2},{\color{blue}4})$&$(1,{\color{blue}3})({\color{blue}2},{4})$
    &$\tr(B_1B_1')\tr(B_2B_2')\tr^2(\Sigma_1\Sigma_2)$\\
   \hline  
$ \begin{array}{c}\xymatrix @-1pc{
 \bullet \ar@/^/@{-}[rr]  & \bullet\ar@{-}[drr]  & \bullet  & \bullet \ar@{-}[dll]
    \\
  \bullet \ar@/^/@{-}[rr]  &\bullet &\bullet &\bullet }\end{array}$
    &5& $ \begin{array}{c}\xymatrix @-1pc{
   \bullet\ar@/^/@{-}[rr]  & \bullet\ar@{-}[dr]  & \bullet   & \bullet\ar@{-}[dlll]
    \\
    \bullet &\bullet\ar@/^/@{-}[rr] &\bullet &\bullet }\end{array}$&$(1,3)({\color{blue}2},{\color{blue}4})$&
    $(1,{\color{blue}3},{2},{\color{blue}4})$
    &$\tr(B_1B_1')\tr(B_2^2)\tr((\Sigma_1\Sigma_2)^2)$\\
\hline   
$ \begin{array}{c}\xymatrix @-1pc{
   \bullet\ar@{-}[drr] & \bullet\ar@{-}[d]  & \bullet  & \bullet \ar@{-}[d]
    \\
  \bullet \ar@{-}[urr]  &\bullet &\bullet &\bullet }\end{array}$
    &0& $ \begin{array}{c}\xymatrix @-1pc{
   \bullet\ar@{-}[drrr]  & \bullet\ar@{-}[dl]  & \bullet\ar@{-}[dl]   & \bullet\ar@{-}[dl]
    \\
    \bullet &\bullet &\bullet &\bullet }\end{array}$&
    $(1,3)({\color{blue}2})({\color{blue}4})$&
    $(1,{\color{blue}4},3,{\color{blue}2})$
    &$\tr(B_1^2)\tr^2(B_2)\tr((\Sigma_1\Sigma_2)^2)$\\
    \hline 
$ \begin{array}{c}\xymatrix @-1pc{
 \bullet\ar@{-}[drr] & \bullet \ar@/^/@{-}[rr]  & \bullet  & \bullet
    \\
  \bullet \ar@{-}[urr] &\bullet \ar@/^/@{-}[rr] &\bullet &\bullet }\end{array}$
    &5& $ \begin{array}{c}\xymatrix @-1pc{
   \bullet\ar@{-}[drrr]  & \bullet\ar@/^/@{-}[rr]  & \bullet\ar@{-}[dl]   & \bullet
    \\
    \bullet\ar@/^/@{-}[rr] &\bullet &\bullet &\bullet }\end{array}$&
    $(1,3)({\color{blue}2},{\color{blue}4})$&$(1,{\color{blue}4},{2},{\color{blue}3})$
    &$\tr(B_1^2)\tr(B_2B_2')\tr((\Sigma_1\Sigma_2)^2)$\\
   \hline  
$ \begin{array}{c}\xymatrix @-1pc{
 \bullet\ar@{-}[drr]  & \bullet\ar@{-}[drr]  & \bullet  & \bullet \ar@{-}[dll]
    \\
  \bullet \ar@{-}[urr]  &\bullet &\bullet &\bullet }\end{array}$
    &4& $ \begin{array}{c}\xymatrix @-1pc{
   \bullet\ar@{-}[drrr]  & \bullet\ar@{-}[dr]  & \bullet\ar@{-}[dl]   & \bullet\ar@{-}[dlll]
    \\
    \bullet &\bullet &\bullet &\bullet }\end{array}$&$(1,3)({\color{blue}2},{\color{blue}4})$&
    $(1,{\color{blue}4})({2},{\color{blue}3})$

    &$\tr(B_1^2)\tr(B_2^2)\tr^2(\Sigma_1\Sigma_2)$\\
     \hline \hline
\end{tabular}
\end{table}
}
Nine possible $\gamma\in\calF_4(t)$ and their contributions are listed in Table \ref{Table1}.
Thus
\begin{multline*}
  \mbox{\rm Var}(\mW_1\mW_2)=(\tr(B_1B_1')\tr(B_2B_2')+\tr(B_1^2)\tr(B_2^2))\tr^2(\Sigma_1\Sigma_2)\\
  +\big(\tr(B_1B_1')\tr(B_2^2)+\tr(B_1^2)\tr(B_2B_2')
  +\tr^2(B_1)\tr(B_2^2)+\tr^2(B_1)\tr(B_2B_2')\\+
  \tr(B_1^2)\tr^2(B_2)+\tr(B_1B_1')\tr^2(B_2)
  \big)\tr((\Sigma_1\Sigma_2)^2).
\end{multline*}
With $B_j=\left[\begin{matrix}
I_{M_j}&0\\
0&0
\end{matrix}\right]$ and $\Sigma_j=I_N$ this gives
$$
\mbox{\rm Var}(\mW_1\mW_2)=2M_1M_2N^2+2M_1M_2(M_1+M_2+1)N.
$$
(See also Example \ref{Ex3}.)
\end{example}
\subsubsection{Remarks}
\begin{remark}
Wishart expansions can also be modeled by planar algebras, see
\cite{guionnet-jones-dima-2007}.
\end{remark}

\begin{remark} Applying \eqref{Def gW} to complex Gaussian $\mX$, one can define the
corresponding ``generalized complex Wishart" matrices with complex-valued $B$ and Hermitian $\Sigma$.
\cite[Theorem 2]{Bryc-06} can be generalized to this setting, and the resulting formula is
\begin{equation}
  \label{Main formula CC}
\E\left(p_{\sigma,t}(\mW_1,\mW_2,\dots,\mW_s)\right)=\sum_{\gamma\in \calF_n^+(t)}
p_{\gamma^{-1},t}(B_1,\dots,B_s)p_{\sigma\circledcirc\gamma,t}(\Sigma_1,\Sigma_2,\dots,\Sigma_s).
\end{equation}
Moreover,  since $\gamma,\sigma\in\calF_n^+$,
we can identify them with the pairs of permutations $\pi(\gamma)$ and $\pi(\sigma)$ respectively;
this identification does  not affect monomials $p_\gamma$ which depend only on the cycle type of
permutation $\pi(\gamma)$.
 When $\eps=(+,\dots,+)$,  we have $q_\gamma=p_\gamma$, and $t(\sigma,\gamma)=t$;
 since the identification maps Brauer product to composition of permutations, see
 \eqref{pi-composition}, we can write this formula directly in terms of permutations:
 \begin{multline*}
    \E\left(p_{\sigma,t}(\mW_1,\mW_2,\dots,\mW_s)\right)=\sum_{\gamma\in \Sn_n(t)}
p_{\gamma^{-1},t}(B_1,\dots,B_s)p_{\sigma\circ\gamma,t}(\Sigma_1,\Sigma_2,\dots,\Sigma_s).
 \end{multline*}

\end{remark}
\begin{proof}
Indeed, in the complex case the only pairings $\gamma$ that contribute to Wick formula
are $\gamma\in\calF_n^+$, so
the proof of \eqref{Main formula} gives
\begin{multline*}
\E\left(p_{\sigma,t}(\mW_1,\mW_2,\dots,\mW_s)\right)=\sum_{\gamma\in \calF_n^+(t)}
q_{\gamma,t}(B_1',\dots,B_s')p_{\sigma\circledcirc\gamma,t}(\Sigma_1,\Sigma_2,\dots,\Sigma_s).
\end{multline*}
Then we note that with $*$ denoting the conjugate transpose, we have
$\tr(AB\dots C)=\overline{\tr(C^*\dots B^* A^*)}$, so  if $c\in\sC(\gamma)$ then
\begin{multline*}
\tr\left(\prod_{j\in c} B'_{t(j)}\right)=\overline{\tr\left(\prod_{j\in c^{-1}} (B'_{t(j)})^*\right)}
=
\overline{\tr\left(\prod_{j\in c^{-1}} \overline{B}_{t(j)}\right)}=
\tr\left(\prod_{j\in c^{-1}}B_{t(j)}\right).
\end{multline*}

\end{proof}

\subsubsection{Special cases of interest}
In this section we list special cases of Theorem \ref{T1}.
Several similar one-matrix cases appeared in the literature,
but without the use of the Brauer product the formulas may look differently at first.

The following specialization of Theorem \ref{T1} shows that  in the real-Wishart case
the Brauer
plays a role analogous to the composition of permutations for the complex-Wishart
case in \cite[Theorem 2]{Bryc-06}.
\begin{corollary}
  \label{T1Ca}
  If $\mW_1\in\WW{\Sigma_1,I_{M_1}},\mW_2\in\WW{\Sigma_2,I_{M_2}},\dots,\mW_s\in\WW{\Sigma_s,I_{M_s}}$
  are independent  and $\sigma\in\calF_n^+$, then
\begin{multline}
  \label{Main formula (A)}
\E\left(p_{\sigma,t}(\mW_1,\mW_2,\dots,\mW_s)\right)
=\sum_{\gamma\in \calF_n(t)}\prod_{j=1}^sM_j^{\#\sC_j(\gamma)}p_{\sigma\circledcirc\gamma,t(\sigma,\gamma)}(\Sigma_1,\Sigma_2,\dots,\Sigma_s).
\end{multline}
\end{corollary}
\begin{proof}Let $M=\max\{M_1,\dots,M_s\}$. We apply Theorem \ref{T1} to $M\times M$ matrices
$$B_j=\left[\begin{matrix}
I_{M_j} & 0 \\
0 & 0
\end{matrix}\right].$$
Since the cycles of $\pi(\gamma)$ preserve $t$, a cycle of color $r$ contributes one factor $\tr(B_{r})=M_{r}$ to $
q_{\gamma,t}(B_1,\dots,B_s)$. Thus  $q_{\gamma,t}(B_1,\dots,B_s)=\prod_{j=1}^sM_j^{\#\sC_j(\gamma)}$.
 \end{proof}

\begin{remark}
For a single real Wishart matrix, Graczyk, Letac \& Massam \cite{Graczyk-Letac-Massam-05} compute moments of
more general monomials
that involve additional arbitrary matrix parameters.
It is plausible that Corollary \ref{T1Ca}
can be deduced from their result by consecutive integration with respect to some of the auxiliary matrix
 parameters.
\end{remark}

The following one-matrix case of Theorem \ref{T1} is a version of \cite[Theorem 3.5]{Hanlon-Stanley-Stembridge-92},
 who did not use  the Brauer product explicitly. (In the one-matrix case $t\equiv 1$ and  $q_{\gamma,t}=p_{\gamma,t}$, so
 we will write $p_{\gamma}$ for both polynomials.)
\begin{corollary}
  \label{T1Cb}
  If $\mW\in\WW{\Sigma,B}$, $B$ is symmetric,  and $\sigma\in\calF_n^+$, then
\begin{equation}
  \label{Main formula (B)}
\E\left(p_{\sigma}(\mW)\right)=
\sum_{\gamma\in \calF_n(t)}
p_{\gamma}(B)p_{\sigma\circledcirc\gamma}(\Sigma).\end{equation}
\end{corollary}

To compute $p_{\gamma}(I_N)$
one only needs to count the number of cycles in $\gamma\in\calF_n$.
So when $\Sigma,B$ are both identity matrices,
 the expression depends only on the number of cycles of the appropriate 2-regular graphs.

\begin{corollary}[{\cite[Corollary 2.4]{Hanlon-Stanley-Stembridge-92}}]
  \label{HLL}
For $\la=(\la_1,\la_2,\dots)\vdash n$, let $p_\la(\mx)=\tr (\mx^{\la_1})\tr (\mx^{\la_2}) \dots$.
If $\mW\in\WW{I_N,I_M}$, and $\sigma_\la\in\calF_n^+$ is given by \eqref{sigma_la}, then
\begin{equation}
  \label{HSS}
  \E\left(p_\la(\mW)\right)=\sum_{\gamma\in \calF_n} M^{\# \sC(\delta\cup\gamma)} N^{\# \sC(\gamma \cup\sigma_\la)}.
\end{equation}
\end{corollary}
\begin{proof}
From the definition of $\pi(\gamma)$ it is clear that
the number of cycles of $\delta\cup \gamma$ matches the number of cycles of
$\pi(\gamma)$.

The cycles of $\pi(\sigma\circledcirc\gamma)$ arise from adding the vertical
edges to connect the first and last row of the three-row diagram. In view of our choice of $\sigma_\la$,
this gives the same cycles as adding to the graph of $\gamma$  the edges from
the graph of $\hat{\sigma}$ defined by $\hat{\sigma}_\la(i)=-\sigma_\la(-i)$, which is just the graph of $\sigma_\la$
with swapped rows. Of course $\pi(\hat{\sigma}_\la)$ is the inverse of $\pi(\sigma_\la)$
so it has the same cycle structure. Thus $p_\la=p_{\hat{\sigma}_\la}$. Applying
 \eqref{Main formula} to $p_{\hat{\sigma}_\la}$, we end with \eqref{HSS}.
\end{proof}


\subsection{Noncommutative version}\label{Sect NC}
To extend Theorem \ref{T1} to noncommutative setting we need additional notation.

 For $\gamma\in\calF_n$,
let $\CR(\gamma)$ denote the number of crossings of $\gamma$ treated as the pair partition of the ordered set
$(1,-1,2,-2,\dots,n,-n)$. (The second column of Table \ref{Table1} lists some examples.)

In order to extend the definition of
 polynomials $p_{\sigma,t}$ to matrix arguments with entries that do not commute we need to specify
 the order in which the traces appear.
To do so, we follow the standard conventions:
\begin{itemize}
  \item we write the cycles of permutation $\pi(\sigma)$ starting with the smallest element. Thus
if $c_k=(j_{k1},j_{k2},\dots,j_{k\ell_k})$ is the $k$-th cycle, then we assume that $j_{k1}\leq j$ for all $j\in c_k$.
\item we order the cycles $c_1,c_2,\dots$ by the order of their first elements. Thus
$1=j_{11}<j_{21}<\dots$.
\end{itemize}
Under these conventions, we define
$$
p_{\sigma,t}(\mW_1,\dots,\mW_s)=\tr(\mW_{t(1)}\mW_{t(j_{12})}\mW_{t(j_{1\ell_1})})
\tr(\mW_{t(j_{21})}\mW_{t(j_{22})}\dots \mW_{t(j_{2\ell_2})})\dots
$$
We will use shorthand notation
$$p_{\sigma,t}(\mW_1,\dots,\mW_s)=\prod_{c\in\sC(\sigma)}\tr\left(\prod_{j\in c}\mW_{t(j)}\right)$$
which should be interpreted in the above sense.

Our next result extends Theorem \ref{T1} to noncommutative setting. We remark  that due to noncommutativity, we allow
only special choices of $\sigma$, and that polynomials $q_{\gamma,t}$ are not needed,
 as matrices $B_1,\dots,B_s$ are symmetric.
\begin{theorem}
  \label{T2}
  Let $B_1,\dots,B_s\in\calM_{M\times M}(\sR)$ and
  $\Sigma_1,\dots,\Sigma_s\in\calM_{N\times N}(\sR)$ be  positive definite matrices.
  Let $\sigma\in\calF_n^+$ be given by \eqref{sigma_la}
  and let $t:\{1,\dots,n\}\to\{1,\dots,s\}$.
  If $\mW_1\in\WWq{\Sigma_1,B_1},\mW_2\in\WWq{\Sigma_2,B_2},\dots,\mW_s\in\WWq{\Sigma_s,B_s}$
  are $q$-orthogonal, $-1\leq q\leq 1$,
  then
\begin{equation}
  \label{Main q-formula}
\EE\left(p_{\sigma,t}(\mW_1,\dots,\mW_s)\right)=\sum_{\gamma\in \calF_{n}(t)}
q^{\CR(\gamma)}p_{\gamma,t}(B_1,\dots,B_s)p_{\sigma\circledcirc\gamma,t(\sigma,\gamma)}(\Sigma_1,\Sigma_2,\dots,\Sigma_s).
\end{equation}
\end{theorem}

The proof is given in Section \ref{Sect Proof 1}. Here we apply \eqref{Main q-formula}
 to deduce a non-asymptotic representation of the Marchenko-Pastur law.
\subsubsection{A matrix representation of the compound Marchenko-Pastur law}
There is a number of results that relate Marchenko-Pastur law to the semicircle law.
The well known result from \cite{Nica-Speicher96} says that if $\XX$ has the semicircle law then
$\XX^2$  has the Marchenko-Pastur law with parameter $\la=1$. (In fact, Ref. \cite{Nica-Speicher96} shows that
for every value of $\la>0$, one can find a projection $P$ such that $\XX P\XX$ has the Marchenko-Pastur law.)
A related representation  arises by
interpreting an $nM \times n N$ matrix of independent random variables as an $M\times N$ block matrix;
 the limit as $n\to\infty$ represents the Marchenko-Pastur law with parameter $M/N$  as $\mX^*\mX$
 with a matrix $\mX$ of
 free circular entries. Yet another representation that uses free circular elements appears in
\cite[Proposition 5.2]{Thorbjornsen-00}.

Our representation
of the Marchenko-Pastur law differs slightly, as our matrix $\XX$ is not self-adjoint and has semicircular entries. Of course, the result is
 not surprising; but an argument  is needed as occasionally
 squares of symmetric and nonsymmetric matrices behave differently, see \cite[Section 4]{mingo-2007}.


We first recall some properties of non-crossing partitions and recall the definition of a compound Marchenko-Pastur
law.

A partition $\calV$ of an ordered set is non-crossing if $\CR(\calV)=0$.
By $NC[1,\dots,n]$ we denote the set of all non-crossing partitions of $(1,\dots,n)$ .

The free cumulants $\{r_k: k\geq 1\}$ of a probability measure $\mu$ with all moments
are defined implicitly by requiring that for all $n\in\NN$,
$$
\int_\RR x^n \mu(dx)=\sum_{\calV\in NC[1,\dots,n]}\prod_{V\in\calV}r_{\#V}.
$$
For more details, see \cite[Section 2.5]{Hiai-Petz}; for the operator-valued case, see \cite{Speicher-98}.

The Marchenko-Pastur law with parameter $\la>0$ is the law
$$
\pi_\la(dx)=(1-\la)^+\delta_0+\frac{\sqrt{4\la-(x-1-\la)^2}}{2\pi x}1_{(1-\sqrt{\la})^2\leq x\leq (1+\sqrt{\la})^2}  dx,
$$
where $\delta_0$ denotes the point mass at $0$.
The free cumulants of the Marchenko-Pastur law are
$r_k=\la$ for all $k\geq 1$.

If $\nu$ is a probability measure on the Borel sets of $\RR$ with finite moments, then the compound Marchenko-Pastur law $\pi_{\nu,\la}(dx)$
is defined as the law with  free cumulants
$$r_k=\la \int_\RR x^k \mu(d x).$$
Thus the $n$-th moment of the compound Marchenko-Pastur law is
\begin{equation}\label{compound MP}
  \int_\RR x^n \pi_{\nu,\la}(dx)=\sum_{\calV\in NC[1\dots n]} \la^{\# \calV}\prod_{V\in \calV} \int_\RR x^{\#V} \mu(dx),
\end{equation}
compare \cite[(3.3.10)]{Hiai-Petz}. In particular, the Marchenko-Pastur law corresponds to $\nu=\delta_1$.

\begin{proposition}\label{T3}  Let $\nu$ be the distribution of the eigenvalues of a real positive-definite  matrix $B$
with respect to the normalized $\tr_M:\calM_{M\times M}\to \RR$. (That is, $\nu$ assigns mass $1/M$ to each of the eigenvalues of $B$).
Let $\la=M/N$.
If $\mW\in\WWq{I_N,B}$ and $q=0$ then $\frac1N\mW$ has compound
Marchenko-Pastur law
 $\pi_{\nu,\la}$ with respect to state $\tau\circ\tr_N$. In particular, if $\mW\in\WWq{I_N,I_M}$ and $q=0$ then $\frac1N\mW$ has Marchenko-Pastur law with parameter
$\la$.
\end{proposition}

A pair partition
$\gamma\in\calF_n$ is non-crossing,
if $\CR(\gamma)=0$. Notice that all non-crossing $\gamma$ are in $\calF_n^+$.
  Indeed, our ordering of $\{1,-1,\dots,n,-n\}$ implies that
 any two numbers of the same sign have an odd number of elements between them. Thus any pair in $\gamma$ of the
 same sign  must generate a crossing.

We will need to know that for a non-crossing $\gamma\in\calF_n$, permutation $\pi(\gamma)$ is the inverse of a
 non-crossing permutation.
Biane \cite{Biane-97} defines non-crossing permutations $\beta\in\Sn_n$ by the following two requirements:
\begin{itemize}
  \item[(a)]  Each cycle of $\beta$
 can be written in increasing
order
  \item[(b)] The partition of $\{1,\dots,n\}$ into the cycles of $\beta$ is non-crossing.
\end{itemize} Biane characterizes non-crossing permutations by
the geodesic condition
\begin{equation}
  \label{Bianes geodesic}
  \#\sC(\beta)+\#\sC(\rho\circ \beta^{-1})=n+1.
\end{equation}
(See \cite[Remark 2.9]{Mingo-Nica-04}, \cite[Section 6]{Kusalik-Mingo-Speicher-05}, and
\cite[Exercise 6.19.hh]{Stanley-99v2}. Note that this is $g=0$ in \eqref{Jacques} below.)

It is well known that the non-crossing pair partitions  of $2n$ elements are in a one-to-one correspondence with
the non-crossing partitions of $n$ elements, see \cite[Exercise 6.19]{Stanley-99v2};
see also \cite[page 74]{Hiai-Petz}. It turns out that the restriction of $\pi:\calF_n\to\Sn_n$ to non-crossing pair partitions
defines an explicit version of this bijection.

\begin{lemma}\label{L-NC}
 Let $\tilde{\pi}$ be the mapping from pair partitions of $\{\pm 1,\dots,\pm n\}$
 to partitions  of $\{1,\dots,n\}$ which assigns to $\gamma\in\calF_n$ the partition into the cycles of
  permutation $\pi(\gamma)$. Then $\tilde{\pi}$ is a bijection from the non-crossing pair partitions
  with respect to ordering $(1,-1,\dots,n,-n)$ to the
  non-crossing partitions of $(1,\dots,n)$ in natural order.
  Moreover,
  for non-crossing  $\gamma\in\calF_n$, all cycles of $\pi(\gamma)$
  can be represented as decreasing sequences.
\end{lemma}
(We note that our "standard" representation of a cycle begins with the smallest element,
which therefore needs to be moved to the end of the cycle to get a decreasing sequence representation.)

The proof below is due to the referee.
\begin{proof}
Non-crossing pair partitions are in $\calF_n^+$,
and we already noted in Remark \ref{Remark: bijection} that $\pi:\calF_n^+\to \Sn_n$ is one-to-one.

Let $\gamma\in\calF_n^+$ and let $\alpha\in \Sn$ be such that $\pi(\gamma)=\alpha$. Notice that permutation
$\gamma\circ (1,-1,2,-2,\dots,n,-n)$ works by sending $i\mapsto \alpha^{-1}(i)$ and
$-i\mapsto -(\alpha\circ(1,2,\dots,n))$, for every $1\leq i\leq n$. Hence
\begin{equation}
  \label{referee}
  \#\sC\left(\gamma\circ (1,-1,2,-2,\dots,n,-n)\right)=
  \#\sC\left(\alpha^{-1}\right)+\#\sC\left(\alpha\circ (1,2,\dots,n)\right).
\end{equation}

By using this fact and Biane's geodesic condition \eqref{Bianes geodesic}, we see that
$\gamma$ is non-crossing with respect to $(1,-2,2,-2,\dots,n,-n)$ iff
$$
\#\sC(\gamma)+\#\sC\left(\gamma\circ (1,-1,2,-2,\dots,n,-n)\right)=2n+1.
$$
Since $\#\sC(\gamma)=n$, from \eqref{referee} we see that this is equivalent to
$$
\#\sC\left(\alpha^{-1}\right)+\#\sC\left(\alpha\circ (1,2,\dots,n)\right)=n+1,
$$
which, again by \eqref{Bianes geodesic}, is equivalent to the statement that $\alpha^{-1}$ is non-crossing.
\end{proof}

\begin{proofof}{Proof of Proposition \ref{T3}}
  When $q=0$, \eqref{Main q-formula} gives
  $$
  N^{-n}\EE\circ\tr_N(\mW^n)=\sum_{\gamma} N^{\#\sC(\sigma\circledcirc\gamma)-n-1}\prod_{c\in\sC(\gamma)}\tr(B^{\#c}),
  $$
  and the sum is taken over the non-crossing pair partitions $\gamma$.
Non-crossing $\gamma$ are in $\calF_n^+$, so we can use \eqref{pi-composition}. With $\rho=\pi(\sigma)=(1,\dots,n)$, we  get
 $$
  N^{-n}\EE\circ\tr_N(\mW^n)=\sum_{\gamma\in\calF_n \mbox{ non-crossing}}
  N^{\#\sC(\rho\circ\pi(\gamma))-n-1}\prod_{c\in\sC(\gamma)}\tr(B^{\#c}) .
  $$

By Lemma \ref{L-NC}, we can take the sum over the permutations $\alpha$ obtained as images of non-crossing $\gamma$.
 \begin{multline*}
    N^{-n}\EE\circ\tr_N(\mW^n)=\sum_{\alpha} N^{\#\sC(\rho\circ\alpha)-n-1}\prod_{c\in\sC(\alpha)}\tr(B^{\#c})
  =\sum_{\alpha} \la^{\#\sC(\alpha)} N^{\#\sC(\alpha)+\#\sC(\rho\circ\alpha)-n-1}\prod_{c\in\sC(\alpha)}\tr_M(B^{\#c}) ,
 \end{multline*}
 If the cycles of $\alpha\in\Sn_n$ form a non-crossing partition and are in decreasing order, then the cycles
of $\beta:=\alpha^{-1}$ are in increasing order and form a non-crossing partition, too.
Since $\alpha$ and $\alpha^{-1}$ have the same
number of cycles,
\eqref{Bianes geodesic} can be written as
\begin{equation}
  \label{Bianes formula}
  \#\sC(\rho\circ\alpha)+\#\sC(\alpha)=n+1.
\end{equation}
Thus,
 \begin{multline*}
    N^{-n}\EE\circ\tr_N(\mW^n)=\sum_{\alpha\in\Sn_n, \alpha^{-1} \mbox{ non-crossing}}
  \la^{\#\sC(\alpha)} \prod_{c\in\sC(\alpha)}\int  x^{\#c}\nu(dx)
  =\sum_{\calV \in NC[1,\dots,n]}  \la^{\#\calV}  \prod_{V\in\calV}\int  x^{\#V}\nu(dx),
 \end{multline*}
 which matches the moments \eqref{compound MP} of the compound
  Marchenko-Pastur law.
\end{proofof}

\section{Proofs of Theorems  \ref{T1} and \ref{T2}}\label{Sect Proof 1}
Following \cite{Mingo-Nica-04}, see \cite[Lemma 1]{Bryc-06}, we write polynomials
$p_{\gamma,t}$ and $q_{\gamma,t}$  as follows.
\begin{lemma}
  \label{L MingoNica}
Let $\gamma\in\calF_n$.  With $\alpha=\pi(\gamma)$, $(\eps_1,\dots\eps_n)=\eps(\gamma)$,
\begin{equation}
  \label{MingoNica}
  p_{\gamma,t}(\mx_1,\dots,\mx_s)=
\sum_{I:\{1,\dots,n\}\to\{1,\dots,N\}}
\prod_{i=1}^n [\mx_{t(i)}]_{I(i),I(\alpha(i))}\;,
\end{equation}
\begin{equation}
  \label{MingoNica2}
  q_{\gamma,t}(\mx_1,\dots,\mx_s)=
\sum_{I:\{1,\dots,n\}\to\{1,\dots,N\}}
\prod_{i=1}^n [\mx_{t(i)}^{-\eps_i}]_{I(i),I(\alpha(i))}.
\end{equation}
\end{lemma}
\begin{proofof}{Proof of Theorem \ref{T1}}\label{OldProofOfT1}
Denote $\rho=\pi(\sigma)\in\Sn_n$. Expanding matrix multiplications in \eqref{MingoNica}, we write the polynomial on the
left hand side of \eqref{Main formula} as
\begin{multline}\label{all cycles}
p_{\sigma,t}(\mW_1,\mW_2,\dots,\mW_s)
=\sum_{I:\{1,\dots,n\}\to\{1,\dots,N\}}
\prod_{i=1}^n [A_{t(i)}'\mX_{t(i)}'B_{t(i)}\mX_{t(i)}A_{t(i)}]_{I(i),I(\rho(i))}
\\
=\sum_{I,U,K,J,V}\left(\prod_{i=1}^n [\mX_{t(i)}']_{U(i),K(i)}[\mX_{t(i)}]_{J(i),V(i)}\right) 
\prod_{i=1}^n
[A_{t(i)}']_{I(i),U(i)}[A_{t(i)}]_{V(i),I(\rho(i))}
 \prod_{i=1}^n
[B_{t(i)}]_{K(i),J(i)}\;,
\end{multline}
where the sum is now taken over  all functions
$I,U, V:\{1,\dots,n\}\to\{1,\dots,N\}$ and over all  functions $K,J:\{1,\dots,n\}\to\{1,\dots,M\}$.
The Wick formula expresses the expectation of the product of jointly normal centered
random variables as the sum  over all pairings of product of the covariances for each pair of random variables,
compare Definition \ref{q-Prob Space} with $q=1$.
To describe the pairings,
we associate
positive integers $j=1,2,\dots,n$ with $\mX_{t(j)}'$ and the negative integers $-j$ with $\mX_{t(j)}$.
Then the  pairings are determined by $\gamma\in\calF_n(t)$;
 denoting by $\alpha=\pi(\gamma)$, Table \ref{Tbl2} lists four types of matches that might
 contribute to the moments.

\begin{table}[h]
\caption{\label{Tbl2} Types of matches along $\gamma\in\calF_n(t)$; here $\alpha=\pi(\gamma)\in\Sn_n$. }\begin{tabular}{|c|c|c|c|}
\hline Match type & graph & matching entries & identification of indices\\ \hline\hline
  A  &
 $ \begin{matrix}\xymatrix  @-1pc{
   {^j_\bullet} \ar@{-}[dr] & {^{\alpha(j)}_\bullet}
    \\
    {^\bullet} &{^\bullet}  }\end{matrix}$ &
    $[\mX_{t(j)}']_{U(j)K(j)}=[\mX_{t(\alpha(j))}]_{J(\alpha(j)) V(\alpha(j))}$ & $U(j)=V(\alpha(j))$, $K(j)=J(\alpha(j))$\\
\hline  B  &
 $ \begin{matrix}\xymatrix  @-1pc{
   {^j_\bullet} \ar@{-}[r] & {^{\alpha(j)}_\bullet}
    \\
    {^\bullet} &{^\bullet}  }\end{matrix}$ &
    $[\mX_{t(j)}']_{U(j)K(j)}=[\mX_{t(\alpha(j))}']_{U(\alpha(j)) K(\alpha(j))}$ & $U(j)=U(\alpha(j))$, $K(j)=K(\alpha(j))$\\
 \hline C &
 $ \begin{matrix}\xymatrix  @-1pc{
   {^j_\bullet}  & {^{\alpha(j)}_\bullet}
    \\
    {^\bullet}\ar@{-}[r] &{^\bullet}  }\end{matrix}$ &
    $[\mX_{t(j)}]_{J(j)V(j)}=[\mX_{t(\alpha(j))}]_{J(\alpha(j)) V(\alpha(j))}$ & $J(j)=J(\alpha(j))$, $V(j)=V(\alpha(j))$\\
\hline  $A^{-1}$  &
 $ \begin{matrix}\xymatrix  @-1pc{
   {^j_\bullet}  & {^{\alpha(j)}_\bullet}\ar@{-}[dl]
    \\
    {^\bullet} &{^\bullet}  }\end{matrix}$ &
    $[\mX_{t(j)}]_{J(j)V(j)}=[\mX_{t(\alpha(j))}']_{U(\alpha(j)) K(\alpha(j))}$ & $J(j)=K(\alpha(j))$, $V(j)=U(\alpha(j))$\\\hline\hline
\end{tabular}
\end{table}
The expectation  becomes
\begin{equation*}
\E\left(\prod_{i=1}^n [\mX_{t(i)}']_{U(i),K(i)}[\mX_{t(i)}]_{J(i),V(i)}\right)=
\sum_{\gamma\in \calF_n(t)} \chi_A(\gamma,U,V)\chi_B(\gamma, K,J).
\end{equation*}
where  $\chi_A(\gamma,U,V)\chi_B(\gamma, K,J)=1$ only when functions
$U,V$ and $K,J,$ fulfill the matches specified in the last column of Table \ref{Tbl2} at each edge of $\gamma$;
otherwise, $\chi_A(\gamma,U,V)\chi_B(\gamma, K,J)=0$.

This allows us to write
$$\E(p_{\sigma,t}(\mW_1,\mW_2,\dots,\mW_s))=\sum_{\gamma\in\calF_n(t)} \Pi_A(\gamma) \Pi_B(\gamma),$$
 where
  \begin{eqnarray*}
\Pi_A(\gamma) &:=&\sum_{I,U,V}\chi_A(\gamma,U,V)\prod_{i=1}^n
[A_{t(i)}']_{I(i),U(i)}[A_{t(i)}]_{V(i),I(\rho(i))},\\
\Pi_B(\gamma)&:=&
\sum_{K,J}\chi_B(\gamma, K,J) \prod_{i=1}^n
[B_{t(i)}]_{K(i),J(i)}.\end{eqnarray*}

It remains to prove the following two results.

\begin{lemma}\label{L-A} If $\gamma\in\calF_n(t)$ then
\begin{equation}
  \label{Pi2A}\Pi_A(\gamma)=p_{\sigma\circledcirc\gamma,t(\sigma,\gamma)}(\Sigma_1,\dots,\Sigma_s).
\end{equation}
\end{lemma}
\begin{lemma}\label{L-B}If $\gamma\in\calF_n(t)$ then
$\Pi_B(\gamma)=q_{\gamma,t}(B_1,\dots,B_s)$.
\end{lemma}

\begin{proofof}{Proof of Lemma \ref{L-B}}

Given $K,J:\{1,\dots,,n\}\to\{1,\dots,M\}$, define $\widetilde{K}:\{\pm 1,\dots,\pm n\}\to \{1,\dots,M\}$ by
\begin{equation}
  \label{tilde(K)}
\widetilde{K}(j)=\begin{cases}
K(j)& \mbox{ if $j>0$}, \\
J(-j) & \mbox{ if $j<0$}.
\end{cases}
\end{equation}
It is clear that $(K,J)\mapsto \widetilde{K}$ is a bijection
between the set of all pairs of functions $K,J:\{1,\dots,,n\}\to\{1,\dots,M\}$ and all functions
 $\widetilde{K}:\{\pm 1,\dots,\pm n\}\to \{1,\dots,M\}$.

 From Table \ref{Tbl2} it follows that
\begin{equation}
  \label{chiB}
  \chi_B(\gamma,J,K)=\prod_{\{u,v\}\in\gamma}1_{\widetilde{K}(u)=\widetilde{K}(v)}.
\end{equation}

For a fixed $\gamma\in\calF_n$,
 denote
$\alpha=\pi(\gamma)$, $(\eps_j)=\eps(\gamma)$.
Then we can factor $\Pi_B(\gamma)$ over the cycles of $\alpha$,
\begin{multline*}
\Pi_B(\gamma)=\sum_{K,J}\chi_B(\gamma, K,J) \prod_{i=1}^n
[B_{t(i)}]_{K(i),J(i)}=\prod_{c\in C(\gamma)}\sum_{\left.K\right|_c,\left.J\right|_c}
\chi_B(\gamma, \left.K\right|_c,\left.J\right|_c)
\prod_{j\in c}[B_{t(j)}]_{K(j),J(j)}.
\end{multline*}
On a single cycle $c=(j_1,\dots,j_r)$ of $\alpha$,  $t:=\left.t\right|_c$ is constant.
So to end the proof, it suffices to show that
  \begin{equation}
    \label{BBB}
    \sum_{K,J:c\to \{1,\dots,M\}}
\chi_B(\gamma, K,J)
\prod_{j\in c}[B_{t}]_{K(j),J(j)}=\tr\left(\prod_{j\in c} B_t^{-\eps_j}\right).
  \end{equation}
We note that cycle $c$ corresponds to the sequence
$$
C=(j_1,\pm j_2),(\mp j_2,\pm j_3),\dots,(\mp j_{r-1},\pm j_r),(\mp j_r,-j_1)$$
of the ordered edges of $\gamma$. The corresponding edges of $\delta$ are
$$
\delta\left|_C\right.=((-j_1,j_1),(\pm j_2,\mp j_2),\dots, (\pm j_r,\mp j_r)).
$$
These two sequences can be written directly in terms of $(\eps_j)$ as
$$
\delta\left|_C\right.=(-j_1,j_1),(-\eps_{j_2} j_2,\eps_{j_2} j_2),\dots, (-\eps_{j_r} j_r,\eps_{j_r} j_r),
$$
$$
C=(j_1,-\eps_{j_2} j_2),(\eps_{j_2} j_2,-\eps_{j_3} j_3),\dots,(\eps_{j_{r-1}} j_{r-1},-\eps_{j_r} j_r),(\eps_{j_r} j_r,-j_1).$$

For fixed $K,J$, the expression under the sum on the
left hand side of \eqref{BBB} can be written as
\begin{multline}
  \chi_B(\gamma, K,J)
\prod_{j\in c}[B_{t}]_{K(j),J(j)}=
\prod_{i=1}^r 1_{\widetilde{K}(\eps_{j_i}j_i)=\widetilde{K}(-\eps_{j_{i+1}}j_{i+1})}[B_t]_{\widetilde{K}(j_i),\widetilde{K}(-j_i)}\\
=\prod_{i=1}^r 1_{\widetilde{K}(\eps_{j_i}j_i)=\widetilde{K}(-\eps_{j_{i+1}}j_{i+1})}
[B_t^{-\eps(j_i)}]_{\widetilde{K}(-\eps_{j_i} j_i),\widetilde{K}(\eps_{j_i} j_i)}\\
=\prod_{i=1}^r 1_{\widetilde{K}(\eps_{j_i}j_i)=\widetilde{K}(-\eps_{j_{i+1}}j_{i+1})}
[B_t^{-\eps(j_i)}]_{\widetilde{K}(-\eps_{j_i} j_i),\widetilde{K}(-\eps_{j_{i+1}} j_{i+1})}.
\end{multline}
(Here we use cyclic convention $j_{r+1}:=j_1$.) Thus the left hand side of \eqref{BBB} is
\begin{multline}
  \sum_{\widetilde{K}:\{\pm j_1,\dots,\pm j_r\}\to\{1,\dots,M\}}
  \prod_{i=1}^r 1_{\widetilde{K}(\eps_{j_i}j_i)=\widetilde{K}(-\eps_{j_{i+1}}j_{i+1})}
[B_t^{-\eps(j_i)}]_{\widetilde{K}(-\eps_{j_i} j_i),\widetilde{K}(-\eps_{j_{i+1}} j_{i+1})}\\
= \sum_{\widetilde{K}:\{-\eps_{j_1} j_1,\dots,-\eps_{j_r} j_r\}\to\{1,\dots,M\}}
  \prod_{i=1}^r[B_t^{-\eps(j_i)}]_{\widetilde{K}(-\eps_{j_i} j_i),\widetilde{K}(-\eps_{j_{i+1}} j_{i+1})}
  =\tr\left(\prod_{j\in c} B_t^{-\eps_j}\right),
\end{multline}
 proving \eqref{BBB}.
\end{proofof}
\begin{proofof}{Proof of Lemma \ref{L-A}}

We begin by rewriting $\Pi_A(\gamma)$ in terms of matrices $\Sigma_1,\dots,\Sigma_s$.
Recall that $\rho=\pi(\sigma)\in\Sn_n$. We claim that
with
$\widetilde{I}:\{\pm 1,\dots,\pm n\}\to\{1,\dots,N\}$ defined by
\begin{eqnarray}\label{I-tilde}
 \widetilde{I}(j)&:=& \begin{cases}
   I(j) & \mbox{ if $j>0$},\\
   I(\rho(-j))& \mbox{ if $j<0$},
 \end{cases}
\end{eqnarray}
and with $t:\{\pm 1,\dots,\pm n\}\to\{1,\dots,s\}$ extended by $t(u):=t(|u|)$,
we have
\begin{equation}
  \label{PiA-Sigma}
\Pi_A(\gamma)=\sum_{I:\{1,\dots,n\}\to\{1,\dots,N\}} \prod_{\{u,v\}\in\gamma}
[\Sigma_{t(u)}]_{\widetilde{I}(u),\widetilde{I}(v)}.
\end{equation}
To prove \eqref{PiA-Sigma}, we re-parameterize each pair of functions $U,V$ by a single function
$\widetilde{U}:\{\pm 1,\dots,\pm n\}\to\{1,\dots,N\}$ defined by
\begin{eqnarray}
  \widetilde{U}(j)&:=& \begin{cases}
    U(j) & \mbox{ if $j>0$}, \\
   V(-j) & \mbox{ if $j<0$}. \\
  \end{cases} \label{U-tilde}
\end{eqnarray}
By inspecting the rows of Table \ref{Tbl2} we see that
$$\chi_A(\gamma,U,V)=\prod_{\{u,v\}\in\gamma} 1_{\widetilde{U}(u)=\widetilde{U}(v)}.$$
For fixed $I$,
$$\sum_{U,V}\chi_A(\gamma,U,V)\prod_{i=1}^n
[A_{t(i)}']_{I(i),U(i)}[A_{t(i)}]_{V(i),I(\rho(i))}
$$
factors over the cycles of $\alpha$ into the product of
$$\sum_{U,V:c\to \{1,\dots,N\}}\chi_A(\gamma, U,V)
\prod_{j\in c}
[A_{t(j)}']_{I(j),U(j)}[A_{t(j)}]_{V(j),I(\rho(j))}.$$
The corresponding set $C$ of ordered edges  of $\gamma$ is
$$
C=(u_1,v_1),(u_2,v_2),\dots,(u_r,v_r),
$$
and, as in the proof of Lemma \ref{L-B}, we have $u_{k+1}=-v_k$, $k=1,\dots,r$ with $u_{r+1}:=u_1>0$.
Of course, $c=(j_1,\dots,j_r)=(|u_1|,\dots,|u_r|)$, and $t$ is constant on $c$.
Thus, using the symmetry of $A$,
\begin{multline}
\prod_{j\in c}
[A_{t(j)}']_{I(j),U(j)}[A_{t(j)}]_{V(j),I(\rho(j))}=
\prod_{k=1}^r [A_t]_{\widetilde{U}(j_k),\widetilde{I}(j_k)}[A_t]_{\widetilde{U}(-j_k),\widetilde{I}(-j_k)}
\\
=\prod_{k=1}^r [A_t]_{\widetilde{I}(-v_k),\widetilde{U}(-v_k)}[A_t]_{\widetilde{U}(v_{k+1}),\widetilde{I}(v_{k+1})}.
\end{multline}

Since $\chi_A=1$ implies $\widetilde{U}(u_k)=\widetilde{U}(v_k)$, we see that
$$
\widetilde{U}(v_{k+1})=\widetilde{U}(-v_{k}).
$$
Thus,
\begin{multline}
 \sum_{U,V:c\to\{1,\dots,N\}} \chi_A(\gamma, U,V)
\prod_{j\in c}
[A_{t(j)}']_{I(j),U(j)}[A_{t(j)}]_{V(j),I(\rho(j))}\\=
\sum_{\widetilde{U}:\{\pm j_1,\dots,\pm j_r\}\to\{1,\dots,N\}}
\prod_{k=1}^r  1_{\widetilde{U}(v_{k+1})=\widetilde{U}(-v_{k})}
[A_t]_{\widetilde{I}(-v_k),\widetilde{U}(-v_k)}[A_t]_{\widetilde{U}(v_{k+1}),\widetilde{I}(v_{k+1})}
\\=
\sum_{\widetilde{U}:\{-v_1,\dots,-v_r\}\to\{1,\dots,N\}}[A_t]_{\widetilde{I}(-v_k),\widetilde{U}(-v_k)}
[A_t]_{\widetilde{U}(-v_k),\widetilde{I}(v_{k+1})}
=\prod_{\{u,v\}\in C} [\Sigma_t]_{\widetilde{I}(u),\widetilde{I}(v)}.
\end{multline}

We now fix $\gamma\in\calF_n(t)$ and consider 
\eqref{PiA-Sigma}, which we want to factor over the cycles
of $\beta=\pi(\sigma\circledcirc\gamma)$. To do so, let $\tau$ be the bijection from pairs
$\{u,v\}\in\gamma$ onto the corresponding pairs $\{r,s\}\in\sigma\circledcirc\gamma$ which
appeared in Definition \ref{Def t-modified}, see Table \ref{Tbl3a}.

Since $\tau$ is a bijection and the cycles of  $\beta$ are in one-to-one correspondence with the
sequences of pairs in $\sigma\circledcirc\gamma$, to prove
\eqref{Pi2A}
it suffices to show that if $c$ is a cycle of $\beta$ and $C$ is the corresponding sequence of pairs in
$\sigma\circledcirc\gamma$, then
\begin{equation}\label{(*)}
  \prod_{\{u,v\}\in\tau^{-1}(C)}[\Sigma_{t(u)}]_{\widetilde{I}(u),\widetilde{I}(v)}=
  \prod_{j\in c}[\Sigma_{t(\sigma,\gamma)(j)}]_{I(j),I\circ\beta(j)}
\end{equation}

\begin{table}[h]
\caption{\label{Tbl3a} Bijection $\tau$ from pairs $\{u,v\}\in\gamma$ to
pairs $\{r,s\}\in\sigma\circledcirc\gamma$. Here 
$\rho=\pi(\sigma)\in\Sn_n$. }\begin{tabular}{|c|c|c|}\hline
Case & $\{u,v\}\in\gamma$ & $\{r,s\}=\tau(\{u,v\})$ \\ \hline\hline
  $A$  &
 $ \begin{matrix}\xymatrix  @-1pc{
   {^u_\bullet} \ar@{-}[drr] & \dots & {_\bullet} &{^{\rho(-v)}_{\hspace{3mm}\bullet}}
    \\
    {_\bullet} & \dots & {_\bullet^v}\ar@{-}[dr]&{_\bullet}
     \\
    {^\bullet}&\dots &{^\bullet}&  {^\bullet} }\end{matrix}$ &
  $ \{r,s\}=\{u,-\rho(-v)\} $
        \\
\hline  B  &
 $ \begin{matrix}\xymatrix  @-1pc{
   {^u_\bullet}\ar@/^/@{-}[rr]& \dots & {^{v}_{\hspace{2mm}\bullet}}
    \\
    {^\bullet}& \dots  &{^\bullet}  }\end{matrix}$ &
  $ \{r,s\}=\{u,v\}$
    \\ \hline
 $C$ &
 $ \begin{matrix}\xymatrix  @-1pc{
  {^{-u}_\bullet} & {_\bullet}
 &\dots & {^{-v}_{\bullet}} & {_{\bullet}}
  \\
{^u_\bullet}\ar@/^/@{-}[rrr]&   {_\bullet} &\dots &{^v_\bullet}\ar@{-}[dr] &
  \\
 {_\bullet}&   {^{-\rho(-u)}_\bullet}\ar@{-}[ul] &\dots &{_\bullet}  &{^{-\rho(-v)}_\bullet}
}\end{matrix}$
 &
 $ \{r,s\}=\{-\rho(-u),-\rho(-v)\}$
 \\ \hline
    \end{tabular}
\end{table}

We prove \eqref{(*)} by showing that each of the factors
on the left hand side of \eqref{(*)} can be rewritten into the corresponding factor
on the right hand side of \eqref{(*)}. To do so, let
$\{u,v\}=\tau^{-1}(\{r,s\})$, where $\{r,s\}\in C$  are such that
$|r|,|s|$ are the consecutive numbers in cycle $c$, i.e., $|s|=\beta(|r|)$.

We note that $\tau$ preserves the number of positive
elements in each pair $\{u,v\}$; thus if $u,v>0$ then $r,s>0$ and similarly if $uv<0$ then $rs<0$;
 Table \ref{Tbl3a} lists all possible cases with the intermediate three-row diagram.

We now consider all possible cases.
\begin{itemize}
  \item[A-1] Suppose $r<0$, $s>0$. Since $\Sigma_{t(u)}=\Sigma_{t(v)}$ is symmetric,
  without loss of generality
  we may assume $u>0, v<0$.
  Then $u=s$ and $r=-\rho(-v)$, so
  \eqref{I-tilde} gives
\begin{multline*}
[\Sigma_{t(u)}]_{\widetilde{I}(u),\widetilde{I}(v)}=
[\Sigma_{t(-v)}]_{\widetilde{I}(v),\widetilde{I}(u)}=[\Sigma_{t\circ\rho^{-1}(|r|)}]_{{I}(|r|),
{I}(s)}=
[\Sigma_{t(\sigma,\gamma)(|r|)}]_{{I}(|r|),{I\circ\beta}(|r|)}.
\end{multline*}

   \item[A-2] Suppose $r>0$, $s<0$. Since $\Sigma_{t(u)}=\Sigma_{t(v)}$ is symmetric, without loss of generality
  we may assume $u>0, v<0$.
 Then $u=r$ and $s=-\rho(-v)$, so
  \eqref{I-tilde} gives
$$
[\Sigma_{t(u)}]_{\widetilde{I}(u),\widetilde{I}(v)}=[\Sigma_{t(|r|)}]_{{I}(r),
{I}(|s|)}=[\Sigma_{t(\sigma,\gamma)(|r|)}]_{{I}(|r|),{I\circ\beta}(|r|)}.
$$
  \item[B] Suppose $r>0$, $s>0$. Then $\{u,v\}=\{r,s\}$ and \eqref{I-tilde} gives
$$
[\Sigma_{t(u)}]_{\widetilde{I}(u),\widetilde{I}(v)}=[\Sigma_{t(r)}]_{{I}(r),
{I}(s)}=[\Sigma_{t(|r|)}]_{{I}(|r|),
{I\circ\beta}(|r|)},
$$

  \item[C] Suppose $r<0$, $s<0$. Then swapping the roles of $u,v$ if necessary,
  we may assume that $r=-\rho(-u)$ and $s=-\rho(-v)$, so  \eqref{I-tilde} gives
$$
[\Sigma_{t(u)}]_{\widetilde{I}(u),\widetilde{I}(v)}=[\Sigma_{t\circ\rho^{-1}(|r|)}]_{{I}(|r|),
{I}(|s|)}=
[\Sigma_{t(\sigma,\gamma)(|r|)}]_{{I}(|r|),{I\circ\beta}(|r|)}.
$$
\end{itemize}

\end{proofof}
\end{proofof}

\begin{proofof}{Proof of Theorem \ref{T2}}
This proof is similar to the proof of Theorem \ref{T1} with \eqref{q-Wick} used instead of
Wick formula. Throughout the proof, we denote $\rho=\pi(\sigma)$, $\beta=\pi(\sigma\circledcirc\gamma)$.
For functions
$I:\{1,\dots,n\}\to\{1,\dots,N\}$ and $J:\{1,\dots,n\}\to\{1,\dots,M\}$,  we denote by
$\widetilde{I},\widetilde{J}$ functions on $\{\pm 1,\dots,\pm n\}$, which are defined by
\eqref{I-tilde} and $ \widetilde{J}(j):= J(|j|)$, respectively.
We also extend $t$ to $\{\pm 1,\dots,\pm n\}$ by $t(s):=t(|s|)$;
we will use the same symbol $t$ for this extension.
From \eqref{sigma_la} and \eqref{q-Wigner} we see that
\begin{multline*}
\EE\left(p_{\sigma,t}(\mW_1,\dots,\mW_s)\right)=\sum_{I,J}
\EE\left([\mX_{t(1)}]_{\widetilde{J}(1),\widetilde{I}(1)}[\mX_{t(1)}]_{\widetilde{J}(-1),\widetilde{I}(-1)}
\dots
[\mX_{t(n)}]_{\widetilde{J}(n),\widetilde{I}(n)}[\mX_{t(n)}]_{\widetilde{J}(-n),\widetilde{I}(-n)}
\right).
\end{multline*}
By \eqref{q-Wick}, this implies
\begin{equation*}
\EE\left(p_{\sigma,t}(\mW_1,\dots,\mW_s)\right)=\sum_{I,J}\sum_{\gamma\in\calF_n}
q^{\CR(\gamma)}\prod_{\{u,v\}\in\gamma}
\EE\left( [\mX_{t(u)}]_{\widetilde{J}(u),\widetilde{I}(u)}[\mX_{t(v)}]_{\widetilde{J}(v),\widetilde{I}(v)}\right).
\end{equation*}
Since
$\EE\left( [\mX_{t(u)}]_{\widetilde{J}(u),\widetilde{I}(u)}[\mX_{t(v)}]_{\widetilde{J}(v),\widetilde{I}(v)}\right)
=0$ when $t(u)\ne t(v)$, we see that the sum can be restricted to the
color-preserving partitions $\gamma\in\calF_n(t)$. Changing the order of summation, and taking into account
\eqref{q-factorization}, we get
\begin{multline*}
\EE\left(p_{\sigma,t}(\mW_1,\dots,\mW_s)\right)=\sum_{\gamma\in\calF_n(t)}
q^{\CR(\gamma)}\sum_{I,J}\prod_{\{u,v\}\in\gamma}
\EE\left( [\mX_{t(u)}]_{\widetilde{J}(u),\widetilde{I}(u)}[\mX_{t(v)}]_{\widetilde{J}(v),\widetilde{I}(v)}\right)
=\sum_{\gamma\in\calF_n(t)}
q^{\CR(\gamma)}\Pi_1(\gamma)\Pi_2(\gamma),
\end{multline*}
where
\begin{equation}
  \label{q-B}
\Pi_1(\gamma)=\sum_{J:\{1,\dots,n\}\to\{1,\dots,M\}} \prod_{\{u,v\}\in\gamma}
[B_{t(u)}]_{\widetilde{J}(u),\widetilde{J}(v)},
\end{equation}
\begin{equation}
  \label{q-Sigma}
\Pi_2(\gamma)=\sum_{I:\{1,\dots,n\}\to\{1,\dots,N\}} \prod_{\{u,v\}\in\gamma}
[\Sigma_{t(u)}]_{\widetilde{I}(u),\widetilde{I}(v)}.
\end{equation}
To end the proof, it remains to show that
\begin{equation}\label{Pi1}
  \Pi_1(\gamma)=p_{\gamma,t}(B_1,\dots,B_s)
\end{equation}
and
\begin{equation}\label{Pi2}
  \Pi_2(\gamma)=p_{\sigma\circledcirc\gamma,t(\sigma,\gamma)}(\Sigma_1,\dots,\Sigma_s).
\end{equation}
The proof of \eqref{Pi1} is similar to the proof of Lemma \ref{L-B} and is omitted.
%
Since the right hand side of \eqref{q-Sigma} is the same as the right hand side of
\eqref{PiA-Sigma}, formula \eqref{Pi2} follows from the proof of Lemma \ref{L-A}.

\end{proofof}

\section{Proof of Theorem \ref{T4}}\label{Sect Proof of T4}

 The identification of the limit in Theorem \ref{T4}  when $q=0$ or $q=1$ is a consequence of
 the following general property of $q$-Wishart matrices.
 \begin{theorem}
   \label{T5}
   Let
   $$\mW^{(N)}_1,\mW^{(N)}_2,\dots,\mW_s^{(N)}, \widetilde\mW^{(N)}_1,\widetilde\mW^{(N)}_2,\dots,
   \widetilde\mW_s^{(N)}\in\WWq{\frac1NI_N,I_M}$$ be $q$-orthogonal with $M=M(N)$ such that $M/N\to\la$.
    Fix a polynomial $Q(x_1,\dots,x_s)$ in noncommutative variables
     and let
$$X_N=\tr(Q(\mW^{(N)}_1,\dots,\mW^{(N)}_s))-\EE(\tr(Q(\mW^{(N)}_1,\dots,\mW^{(N)}_s)))$$
 and $$Y_N=\tr(Q(\widetilde\mW^{(N)}_1,\dots,\widetilde\mW^{(N)}_s))-\EE(\tr(Q(\widetilde\mW^{(N)}_1,\dots,\widetilde\mW^{(N)}_s))).$$

 Then for $m\geq 0$
   the limits
   $\lim_{N\to\infty} \EE(X_N^2)$, $\lim_{N\to\infty} \EE\left((X_N+Y_N)^m\right)$ exist, and
   \begin{equation}
     \label{cond variance}
     \lim_{N\to\infty}\left(\EE\left((X_N-Y_N)^2 (X_N+Y_N)^m\right) - 2\EE(X_N^2)
     \EE\left((X_N+Y_N)^m\right)\right)=0.
   \end{equation}
 \end{theorem}

\begin{remark}
  \label{R-Bozejko-B}
  Since the conditional expectation of $X_N$ given $X_N+Y_N$ is  $(X_N+Y_N)/2$,
expression \eqref{cond variance} can be interpreted as the statement that
 the conditional variance of $X_N$ given $X_N+Y_N$ is asymptotically constant. In the self-adjoint case,
 the latter property identifies the distribution
 when $q=0,1$, see \cite[Theorem 3.1]{Bozejko-Bryc-04}. We note that
 Theorem \ref{T5} holds true for all $-1\leq q\leq 1$,
 but  Examples \ref{T4-counterexample}, \ref{Ex3}, and \ref{Ex 4} show the limiting law may fail to be $q$-Gaussian for $0<|q|<1$.
  This suggests that there is no
natural  $q$-version of \cite{Bozejko-Bryc-04}.
\end{remark}

\subsection{Auxiliary combinatorial result}
The proof of Theorem \ref{T5} relies on the following  partition version of a theorem of
Jaques \cite{Jacques-68}.
\begin{proposition}
  \label{L:machi}
If $\gamma\in\calF_n$ and $\sigma\in\calF_n^+$ are such that the graph $\gamma\cup\sigma\cup\delta$
is connected, then
there exists an integer $h\geq 0$ such that
\begin{equation}
  \label{eq:machi}
  \#\sC(\gamma)+\#\sC(\sigma)+\#\sC(\sigma\circledcirc\gamma)-n=2-h.
\end{equation}
\end{proposition}
For the proof, we will need to consider $\calF_n$ in all three roles: as pair partitions, as graphs,
and as a subset of $\Sn_{2n}$. To help distinguish between these interpretations,
denote by $z(\alpha)$ the number of cycles of a permutation  $\alpha$.
Recall that in this notation, if $\gamma\in\calF_n$ then
 $\#\sC(\gamma)=z(\pi(\gamma))$.
We will use the above mentioned result of Jaques, which we now state in the version most convenient for our purposes.
\begin{knowntheorem}\cite{Jacques-68}
If $\alpha,\beta\in\Sn_n$ act transitively on $\{1,\dots, n\}$,
 then
there exists an integer $g\geq 0$ such that
\begin{equation}
  \label{Jacques}
  z(\alpha)+z(\beta)+z(\beta^{-1}\alpha)-n=2-2g.
\end{equation}
\end{knowntheorem}
(For an  accessible algebraic proof see \cite[Theorem 1]{Machi-84}.
For a graph-theoretic proof, see \cite[Proposition 1.5.3]{Lando-Zvonkin-04}.)
\begin{proofof}{Proof of Proposition \ref{L:machi}} The proof consists of two claims, and an application of \eqref{Jacques}.
As noted in the proof of \cite[Lemma 3.5]{Collins-Sniady-06},
\begin{claim}\label{Claim1}
For  $\gamma\in\calF_n$,
\begin{equation}
z(\gamma\circ\delta)=2\#\sC(\gamma).
\end{equation}
\end{claim}
(This is because every cycle of $\delta\cup\gamma $ splits into two cycles of $\gamma\circ\delta$.
Recall that $\delta$ is defined by \eqref{eq:delta}.)

We now note the following.
\begin{claim}
  \label{Claim2}
  If $\gamma\in\calF_n$, $\sigma\in\calF_n^+$ and the graph $\gamma\cup\sigma\cup\delta$ is connected,
  then the group generated by permutations $\gamma\circ\delta,\sigma\circ\delta$ has at most two orbits on
  $\{\pm 1,\dots,\pm n\}$.
\end{claim}
To prove this claim, we note that for $\sigma\in\calF_n^+$ the cycles $\sigma\circ\delta$ are exactly
the cycles
$c_1,\dots,c_k$ of $\pi(\sigma)\in\Sn_n$ and $-c_1,\dots,-c_k$.

Since the graph $\gamma\cup\sigma\cup\delta$ is connected, $\gamma$ must connect the sets
$$c_1\cup -c_1, \, c_2\cup -c_2, \dots, c_k\cup -c_k\;,$$
so $\gamma\circ\delta$ must also connect these sets.
Thus if $S$ is an orbit of $\langle\sigma\circ\delta,\gamma\circ\delta\rangle$, then
$\{|j|:j\in S\}=\{1,\dots,n\}$.


As the mapping $S\ni j\mapsto |j|\in\{1,\dots, n\}$ is two-valued and onto, there cannot be more than two disjoint
inverse images of $\{1,\dots, n\}$.

By Claim \ref{Claim2}, applying \eqref{Jacques} to $\alpha=\gamma\circ\delta$,
$\beta=\sigma\circ\delta\in\Sn_{2n}$ or to their restrictions on the two orbits separately,
we have
\begin{equation*}
  z(\gamma\circ\delta)+z(\sigma\circ\delta)+z((\sigma\circ\delta)^{-1}\circ\gamma\circ\delta)-2n=\begin{cases}
    2-2g & \mbox{if $\langle\sigma\circ\delta,\gamma\circ\delta\rangle$ acts transitively }\\
    4-2g_1-2g_2 & \mbox{otherwise}
  \end{cases}
\end{equation*}
for some integers $g,g_1,g_2\geq 0$. Thus in both cases, there is an integer $h\geq 0$ such that
\begin{equation}
  \label{J23}
  z(\gamma\circ\delta)+z(\sigma\circ\delta)+z((\sigma\circ\delta)^{-1}\circ\gamma\circ\delta)-2n=4-2h.
\end{equation}
We now note that
$(\sigma\circ\delta)^{-1}=\hat\sigma\circ\delta$, where $\hat\sigma(j)=-\sigma(-j)$.
Moreover, $\hat\sigma\circ\delta=\delta\circ\hat\sigma$ so permutation
$=(\sigma\circ\delta)^{-1}\circ\gamma\circ\delta=\hat\sigma\circ\delta\circ\gamma\circ\delta$ is similar to $\hat\sigma\circ\gamma$.

In the proof of Corollary \ref {HLL} we already noted that
$\#\sC(\sigma\circledcirc\gamma)$ is the number of components of the graph $\hat\sigma\cup\gamma$.
We now observe that the cycles of $\hat\sigma\circ\gamma$ split each such component in two.
Therefore
$$
z((\sigma\circ\delta)^{-1}\circ\gamma\circ\delta)=z(\hat\sigma\circ\gamma)=2\#\sC(\sigma\circledcirc\gamma).
$$
This, together with Claim \ref{Claim1} and \eqref{J23} gives \eqref{eq:machi}.
\end{proofof}
\subsection{Proofs of Theorems \ref{T5} and \ref{T4}}
Our first goal is to show that the limits of joint moments exist;
 the main step consists of analyzing moments of the
 monomials.

 We will need additional notation.
By $\calF_n(t,\sigma)$ we denote the set of all $\gamma\in\calF_n(t)$ which connect
every cycle of $\sigma$ to some other cycle.
That is, all components of the graph $\delta\cup\sigma\cup\gamma$ are strictly larger than the components of $\sigma\cup\delta$.

\begin{lemma}
  \label{Lemma M1}
  Let  $\mW^{ }_1,\mW^{ }_2,\dots,\mW_s^{ }\in\WWq{\frac1NI_N,I_M}$ be $q$-orthogonal.
Suppose that the cycles $c_u$ ($1\leq u\leq r$) of $\sigma\in\calF_n^+$  consist
 of consecutive integers, as in
 \eqref{sigma_la}.
Fix $t:\{1,\dots,n\}\to \{1,\dots, s\}$.
Then
  \begin{equation}\label{eq M1}
    \EE\left(\prod_{u=1}^r\left( \tr(\prod_{j\in c_u} \mW_{t(j)})-
    \EE\big(\tr(\prod_{j\in c_u} \mW_{t(j)})\big)\right)\right)
    =\sum_{\gamma\in\calF_n(t,\sigma)}q^{\CR(\gamma)}M^{\#\sC(\gamma)}N^{\#\sC(\sigma\circledcirc\gamma)-n}.
  \end{equation}
\end{lemma}
\begin{proof}
Under the standard convention that empty product is identity, the left hand side of \eqref{eq M1} is
\begin{equation}
  \label{M1.1}
\sum_{F\subset \{1,\dots,r\}}(-1)^{\#F}\; \EE\left(\prod_{u\not\in F}
\tr(\prod_{j\in c_u} \mW_{t(j)})\right)\prod_{u\in F} \EE\left(\tr(\prod_{j\in c_u} \mW_{t(j)})\right).
\end{equation}
(Here, finite subsets $F\subset \{1,\dots,r\}$ and their complements in $\{1,\dots,r\}$
are treated as ordered sets with the natural order on integers,
so that the noncommutative products of traces that occur in the first factor are well defined. That is, if
$F=\{u_1,u_2,\dots,u_m\}$ with $u_1<u_2<\dots<u_m$
then
$\prod_{u\in F}x_u:=x_{u_1}x_{u_2}\dots x_{u_m}$ is well defined.)

We now re-rewrite the right hand side of \eqref{eq M1}. Let  $c_1,\dots,c_r$ be the cycles of $\sigma$ in increasing order, and let $k_1,\dots,k_r$ denote their lengths.
(Recall our convention that $c_1,\dots, c_r$ are in fact the cycles of $\pi(\sigma)\in\Sn_n$.)

For $1\leq u\leq r$, let $\calG_u$ denote the set of $\gamma\in\calF_n(t)$ which "isolate" cycle $c_u$, i.e.,
do not connect cycle $c_u$ to any other cycle.
Then $\calF_n(t,\sigma)=\calF_n(t)\setminus\bigcup_{u=1}^r\calG_u$, so
for any  weighted counting measure $H$,
$$
H(\calF_n(t,\sigma))=H(\calF_n(t))- H(\bigcup_{u=1}^r\calG_u)=\sum_{F\subset \{1,\dots,r\}}(-1)^{\#F}
H(\bigcap_{u\in F}\calG_u),
$$
where $H(\calF_n(t))$ corresponds to $F=\emptyset$ in the expansion.
Therefore, the right hand side of \eqref{eq M1} is
$$\sum_{F\subset \{1,\dots,r\}}(-1)^{\#F}
\sum_{\gamma\in\bigcap_{u\in F}\calG_u}q^{\CR(\gamma)}M^{\#\sC(\gamma)}N^{\#\sC(\sigma\circledcirc\gamma)-n}.
$$
Note that $\gamma\in\bigcap_{u\in F}\calG_u$ can be uniquely split into the disjoint union
$\gamma_0\cup\bigcup_{u\in F}\gamma_u$, where $\gamma_u\in\calF_{k_u}$ is a pair partition of cycle $c_u$,
 and $\gamma_0\in \calF_{n_0}$ with $n_0=n-\sum_{u\in F} k_u$.
We now note that $\sigma\circledcirc\gamma=(\sigma_0\circledcirc\gamma_0)
\cup\bigcup_{u\in F}c_u\circledcirc\gamma_u$. Since the cycles of $\sigma$ consist of consecutive segments,
no additional crossings of $\gamma$
can arise from the edges connecting $c_u$ besides the crossings already accounted for as crossings of $\gamma_u$.
Thus $\CR(\gamma)=\CR(\gamma_0)+\sum_{u\in F} \CR(\gamma_u)$, and
 the right hand side  of \eqref{eq M1} can be written as
 \begin{multline*}
 \sum_{F\subset \{1,\dots,r\}}(-1)^{\#F}
\sum_{\gamma_0\in\calF_{n_0}}q^{\CR(\gamma_0)}M^{\#\sC(\gamma_0)}
N^{\#\sC(\sigma_0\circledcirc\gamma_0)-n_0}
\prod _{u\in F}\sum_{\gamma_u\in\calF_{k_u}}q^{\CR(\gamma_u)}M^{\#\sC(\gamma_u)}
N^{\#\sC(c_u\circledcirc\gamma_u)-k_u}.
 \end{multline*}
Applying \eqref{Main q-formula} to each of the factors, we arrive at \eqref{M1.1}, thus deriving the left hand side of
\eqref{eq M1}.
\end{proof}
Next, we show that the non-negligible terms in expansion \eqref{eq M1} correspond to
$\gamma$ that connect pairs of cycles of $\sigma$.

\begin{lemma}
  \label{L M2}
Fix $\gamma\in\calF_n(t,\sigma)$. Then
the limit   \begin{equation}
    \label{limit 3}
    M_n(\gamma):=\lim_{N\to\infty, M/N\to\la}M^{\#\sC(\gamma)}N^{\#\sC(\sigma\circledcirc\gamma)-n}
  \end{equation} exists.
Moreover,  if the number of cycles of $\sigma$
connected by $\gamma$ is more than 2, then $M_n(\gamma)=0$.
 If $\gamma$ connects only pairs $c_u,c_v$ of cycles of $\sigma$, and
 $\gamma_{u,v}$ denotes the corresponding part of $\gamma$, then
$$M_n(\gamma)=\prod_{\{u,v\}}M_2(\gamma_{u,v}).$$
(The product is taken over
all unordered pairs $u,v$ corresponding to the connected cycles of $\sigma$).
\end{lemma}
\begin{proof}
 Since $\gamma\in\calF_n(t,\sigma)$, once we
 decompose $\sigma\cup\gamma\cup\delta$ into connected components, the decomposition splits
 the cycles of $\sigma$ into groups of $m_1,\dots,m_r\geq 2$ cycles,
perhaps non-consecutive. This partitioning
 splits $\gamma, \sigma$ and $\sigma\circledcirc\gamma$ into disjoint
 triplets $\gamma_u, \sigma_u,\sigma_u\circledcirc\gamma_u\in\calF_{n_u}$, $u=1,2,\dots,r$,
 with $n=n_1+\dots+ n_r$,
which contribute multiplicatively,
$$M^{\#\sC(\gamma)}N^{\#\sC(\sigma\circledcirc\gamma)-n}=(M/N)^{\#\sC(\gamma)}
\prod_{u=1}^r N^{\#\sC(\gamma_u)+\#\sC(\sigma_u\circledcirc\gamma_u) +\#\sC(\sigma_u)- n_u-m_u}.$$
(Recall that $m_u=\#\sC(\sigma_u)$.)
Since the graphs of $\sigma_u\cup\gamma_u\cup \delta_u$ are connected, by formula \eqref{eq:machi},
there are integers $h_1,\dots h_r\geq 0$ such that
$$M^{\#\sC(\gamma)}N^{\#\sC(\sigma\circledcirc\gamma)-n}=(M/N)^{\#\sC(\gamma)}
\prod_{u=1}^r N^{2-m_u-h_u}.$$
Therefore, the limit always exists. The limit is zero if at least one of the $m_u>2$ or at least one of the $h_u>0$.
If all groups  are pairs, i.e., $m_1=m_2=\dots=m_r=2$,
then the
limit factors into the contributions from the pairs of cycles of $\sigma$; of course, the only $\gamma$
that contribute to this limit are those with $h_{1}=h_2=\dots=h_r=0$.
\end{proof}
\begin{proofof}{Proof of Theorem \ref{T5}}
 Since $(x-y)^2=x(x-y)+y(y-x)$, and joint moments of $(X_N,Y_N)$ are symmetric in $X,Y$,
it is enough to show that
\begin{equation}
  \label{Const Var 1}
  \lim_{N\to\infty}\left(\EE\left(X_N(X_N-Y_N) (X_N+Y_N)^m\right) - \EE(X_N^2) \EE\left((X_N+Y_N)^m\right)\right)=0.
\end{equation}
After re-indexing the $q$-Wishart matrices so that $\mW_{-j}:=\widetilde\mW_j$,
 we can represent $X_N=\sum \alpha_u X_{N,u}$ and
$Y_N=\sum \alpha_u Y_{N,u}$ as finite linear combinations (with the same coefficients, and the same
choice of functions $t_u:c_u\to\{1,\dots,s\}$)
 of the centered monomials
\begin{eqnarray}
  X_{N,u}&=&\tr(\prod_{j\in c_u} \mW_{t_u(j)})-
    \EE\big(\tr(\prod_{j\in c_u} \mW_{t_u(j)})\big),\label{Const Var1.1a}
\\
Y_{N,u}&=&\tr(\prod_{j\in c_u} \mW_{-t_u(j)})-
    \EE\big(\tr(\prod_{j\in c_u} \mW_{-t_u(j)})\big).\label{Const Var1.1b}
\end{eqnarray}
(Here $c_u$ are finite sets of integers of prescribed cardinality,
 which will be assembled into the cycles of $\sigma$ in the next step.)

Expanding \eqref{Const Var 1} into the monomials,
 it is enough to show that
\begin{equation}
  \label{Const Var 2}
 \lim_{N\to\infty}\Big(
 \EE\big(X_{N,v_0}(X_{N,u}-Y_{N,u}) \prod_{j=1}^m (X_{N,v_j}+Y_{N,v_j})\big)
 -  \EE(X_{N,v_0}X_{N,u})\EE\big(\prod_{j=1}^m (X_{N,v_j}+Y_{N,v_j})\big)
 \Big) =0
\end{equation}
for all admissible indexes $v_0,\dots,v_m,u$. Once $v_0,u,v_1,\dots,v_m$ are fixed,
we modify $t$ and the sets
$c_u$ so that $c_{v_0},c_u,c_{v_1},\dots,c_{v_m}$ are consecutive  blocks of integers that form the cycles of
$\sigma\in\calF_n^+$.
This allows us to apply Lemmas \ref{Lemma M1} and \ref{L M2}, which imply that the limits of moments exist.
Expanding the products
\begin{equation}
  \label{Const Var 2a}
  \prod_{j=1}^n (X_{N,v_j}+Y_{N,v_j})=
\sum_{\eta_1,\dots,\eta_n=0,1}\prod_{j=1}^m X_{N,v_j}^{\eta_j}Y_{N,v_j}^{1-\eta_j}
\end{equation}
we see that it suffices to show that
\begin{multline}
  \label{const Var 3}
\lim_{N\to\infty}\sum_{\eta_1,\dots,\eta_m=0,1}
\EE\left(X_{N,v_0}(X_{N,u}-Y_{N,u})\prod_{j=1}^m X_{N,v_j}^{\eta_j}Y_{N,v_j}^{1-\eta_j}\right)
\\=\lim_{N\to\infty}\sum_{\eta_1,\dots,\eta_m=0,1}\EE\big(X_{N,v_0}X_{N,u}\big)
\EE\big(\prod_{j=1}^m X_{N,v_j}^{\eta_j}Y_{N,v_j}^{1-\eta_j}\big).
\end{multline}
To shorten the notation, denote by  $c_0, c, c_1,\dots,c_m$
 the cycles $c_{v_0},c_u,c_{v_1},\dots,c_{v_m}$ of $\sigma$.
 In view of (\ref{Const Var1.1a}-\ref{Const Var1.1b}) Lemmas \ref{Lemma M1} and \ref{L M2}
imply
 that for each choice of
$\eta=(\eta_1,\dots,\eta_m)\in\{0,1\}^n$  there is $t_{\eta}:\{1,\dots,\}\to\{\pm 1,\dots,\pm s\}$
such that
\begin{equation}
  \label{Const Var 4}
  \EE\big(X_{N,v_0}X_{N,u}\prod_{j=1}^m X_{N,v_j}^{\eta_j}Y_{N,v_j}^{1-\eta_j}\big)=
  \sum_{\gamma\in\calF_n(t_\eta,\sigma)}q^{\CR(\gamma)}M_{m+2}(\gamma)+ O(1/N).
\end{equation}
For $j=0,\dots,m$ denote by $\calH_j(t_\eta)$ the set of $\gamma\in\calF_n(t_\eta,\sigma)$ which
connect cycle $c$ to cycle $c_j$ and do not connect the pair $c,c_j$ to any other cycles of $\sigma$.
Lemma \ref{L M2} implies that \eqref{Const Var 4} can be expanded as
\begin{equation}
  \label{Const Var 4a}
  \EE\big(X_{N,v_0}X_{N,u}\prod_{j=1}^m X_{N,v_j}^{\eta_j}Y_{N,v_j}^{1-\eta_j}\big)=
  \sum_{\gamma\in\calH_0(t_\eta)}q^{\CR(\gamma)}M_{m+2}(\gamma) +
  \sum_{j=1}^m\sum_{\gamma\in\calH_j(t_\eta)}q^{\CR(\gamma)}M_{m+2}(\gamma)+O(1/N).
 \end{equation}
All $\gamma\in\calH_0(t_\eta)$ can be decomposed into $\gamma_0\cup\gamma'$, where
$\gamma_0$ pairs the cycles $c_0,c$ only, and $\gamma'$ is the pairing of the elements underlying cycles
$c_1,\dots,c_n$.
Furthermore, since $c_0,c,c_1,\dots,c_m$ are consecutive blocks of integers,
$\CR(\gamma)=\CR(\gamma_0)+\CR(\gamma')$. Therefore
$q^{\CR(\gamma)}M_{m+2}(\gamma)=q^{\CR(\gamma_0)}M_2(\gamma_0)q^{\CR(\gamma')}M_m(\gamma')$ and
Lemmas \ref{Lemma M1} and \ref{L M2} together imply that the first sum simplifies,
\begin{multline}
  \label{Const Var 5}
  \EE\big(X_{N,v_0}X_{N,u}\prod_{j=1}^m X_{N,v_j}^{\eta_j}Y_{N,v_j}^{1-\eta_j}\big)
  = \EE\big(X_{N,v_0}X_{N,u}\big)\EE\big(\prod_{j=1}^m X_{N,v_j}^{\eta_j}Y_{N,v_j}^{1-\eta_j}\big)\\
  +\sum_{j=1}^m\sum_{\gamma\in\calH_j(t_\eta)}q^{\CR(\gamma)}M_{m+2}(\gamma)+O(1/N).
\end{multline}

For $i\in\{1,\dots,n\}$, let
$$\tilde t_\eta(i)=\begin{cases}
 - t_\eta(i)& \mbox{ if $i\in c$},\\
t_\eta(i)& \mbox{otherwise}.
\end{cases}$$
Since $Y_{N,u}$ differs from $X_{N,u}$ only in the sign of $t|_c$, and since
$\calH_0(\tilde t_\eta)=\emptyset$, the same reasoning as above gives
\begin{equation}
  \label{Const Var 6}
  \EE\big(X_{N,v_0}Y_{N,u}\prod_{j=1}^m X_{N,v_j}^{\eta_j}Y_{N,v_j}^{1-\eta_j}\big)=
 \sum_{j=1}^m\sum_{\gamma\in\calH_j(\tilde t_\eta)}q^{\CR(\gamma)}M_{m+2}(\gamma)+O(1/N).
\end{equation}
In view of \eqref{Const Var 5} and \eqref{Const Var 6}, to prove \eqref{const Var 3}, it is enough to show that
for every $j\geq 1$,
\begin{equation}
  \label{const Var 7}
\sum_\eta \sum_{\gamma\in\calH_j(t_\eta)}q^{\CR(\gamma)}M_{m+2}(\gamma)
 = \sum_\eta \sum_{\gamma\in\calH_j(\tilde t_\eta)}q^{\CR(\gamma)}M_{m+2}(\gamma).
\end{equation}
If $\gamma\in\calH_j(t_\eta)$ then \eqref{Const Var 2a} gives
$$ M_{m+2}(\gamma)=\begin{cases}
  0 & \mbox{ if $\eta_j=0$}, \\
 M_2(\gamma_j)M_m(\gamma'_j) &\mbox{ if $\eta_j=1$},
\end{cases}
$$
where $\gamma_j$ is the pairing on $c\cup c_j$ and $\gamma'_j$ is the remaining set of pairs.
On the other hand, if $\gamma\in\calH_j(\tilde t_\eta)$ then
$$M_{m+2}(\gamma)=\begin{cases}
  0 & \mbox{ if $\eta_j=1$}, \\
  M_2(\gamma_j)M_m(\gamma'_j) & \mbox{ if $\eta_j=0$},
\end{cases}
$$
as the joint moments of $X_{N,u},X_{N,c_j}$ are the same as the moments of $Y_{N,u},Y_{N,c_j}$
and $t=\tilde t$ outside of $c\cup c_j$. The later also implies that with $\eta'=\eta$ except for one entry
$\eta'(j)=1-\eta(j)$, we have
$$\calH_j(t_\eta)=\calH_j(\tilde t_{\eta'}).$$
Thus summing over $\eta'$ on  the left hand side of \eqref{const Var 7} we get the
right hand side, ending the proof.
\end{proofof}
\begin{proofof}{Proof of Theorem \ref{T4}}
If $\la=0$ then all moments of $Z_{M,N}$ converge to zero, which we
will consider as a  normal or a semicircle
 law of variance $0$.
 Through the remainder of the
proof we assume that $\la>0$. Since
$$\left(Q(\mW_1,\dots,\mW_s)-\EE(Q(\mW_1,\dots,\mW_s))\right)^m$$
can be written as a linear combination of centered monomials of the form that appears on the left hand side of \eqref{eq M1},
combining Lemmas \ref{Lemma M1} and \ref{L M2}, we see that all moments converge.
Furthermore, from \eqref{Main q-formula} we see that the $n$-th moment grows at the rate bounded by a multiple of
$\#\calF_n=1\cdot 3\cdot\ldots\cdot(2n-1)$, so in the self-adjoint case the limiting law  is
uniquely determined by moments.

When  $q=1$ and $Q$ is real symmetric, the two copies $X_N,Y_N$ of $Z_{M,N}$ are independent and converge jointly to a pair
of independent identically distributed  self-adjoint (i.e., real) random variables $X,Y$. From Theorem \ref{T5}, we see that
$\E\left((X-Y)^2p(X+Y)\right)=2\E(X^2)\E\left(p(X+Y)\right)$ for all polynomials $p$.  By the previous estimate for the moments,
the characteristic function of
 $X+Y$ is analytic at $0$, so the moment problem is unique and
 we can replace polynomial $p$ by any bounded measurable function. Thus
 $\rm{Var}(X|X+Y)=const$, which is known to imply that $X$ is normal,
 see  \cite[Corollary 4.1]{Laha-57}.

 When $q=0$ and $Q$ is real symmetric, the two copies $X_N,Y_N$ of $Z_{M,N}$ are free and converge jointly to a pair
of free identically distributed  self-adjoint noncommutative random variables $X,Y$. From Theorem \ref{T5}, we see that
$\EE\left((X-Y)^2p(X+Y)\right)=2\EE(X^2)\EE \left(p(X+Y)\right)$ for all polynomials $p$.  For $q=0$, the moments of $X$
are bounded by the number of non-crossing
elements in $\calF_n$; this is the $2n$-th Catalan number, so
 the support of $X+Y$ is bounded and we deduce that
 $\rm{Var}(X|X+Y)=const$, which is  known to imply that $X$ has semicircle law,
 see \cite[Theorem 3.2]{Bozejko-Bryc-04} for a more general result.

\end{proofof}

\subsection{Concluding examples}
The following computer-assisted calculations give low order moments of some polynomials in $q$-Wishart matrices. Together, they illustrate that the limit law depends on the polynomial in a nontrivial way.
 \begin{example}\label{T4-counterexample}
In the simplest case $\mW_N\in\WWq{\frac1N I_N,I_M}$ and $M=\la N$,
we have $\EE(\tr(\mW_N))=\la N$. The first six moments for the
limit of $X_N=\tr(\mW_N)-N\la$  are consistent with the $q^4$-Gaussian law:
with $s^2:=\EE(X_N^2)=(1+q)\la$,  we have $\EE(X_N^3)\to 0$,
$$
\EE(X_N^4)/ s^4\to(2+q^4),\;\EE(X_N^5)\to 0;\;
\EE(X_N^6)/s^6\to 5 + 6\,q^4 + 3q^8 + q^{12} .
$$
\end{example}

\begin{example}\label{Ex 4}
To see whether one could identify the limit law for a second-order polynomial in $\mW\in\WWq{\frac1NI_N,I_M}$,
we computed the first moments of the limit as $N\to\infty$ and $M/N\to\la$ of
$$X_N=\tr(\mW^2)-a\tr(\mW) - \EE(\tr(\mW^2)-a\tr(\mW)).$$
The general expression for the asymptotic variance is cumbersome, but it
simplifies to $\EE(X_N^2)\to \lambda ^2 (q^6+ q^4+ q^2+1)$ when $a=1+q^2+2\la$.
For this value of $a$,  the normalized fourth moments
$\EE\left(X_N^4\right)/(\EE(X_N^2))^2$ converge to $2+q^{16}$, which is
 the fourth moment of the $q^{16}$-Gaussian law of unit variance.
\end{example}
\begin{example}\label{Ex3}
  Suppose $\mW_1,\mW_2\in\WWq{\frac1NI_N,I_M}$ are $q$-orthogonal. Then $\EE(\tr(\mW_1\mW_2))=M^2/N$, see Example
  \ref{Ex1}, and from Table \ref{Table1} we see that
 ${\rm Var}(\tr(\mW_1\mW_2))=(1+q^2)q^4M^2/N^2+2(1+q)M^3/N^3+(1+q)q^4/N^3$.
Denote $X_N=\tr(\mW_1\mW_2)-N\la^2$, and suppose $M=M(N)=\la N+O(1/N)$.
Then $\lim_{N\to\infty}\EE(X_N^2)={\lambda }^2\left( q^4 + q^6 + 2\lambda  + 2q\lambda  \right)$, and
computer-assisted calculations give
  $\lim_{N\to\infty}\EE(X_N^3)=0$ and
  $$\lim_{N\to\infty}\EE(X_N^4)=
  {\lambda }^4\,\left( q^8\,{\left( 1 + q^2 \right) }^2\,\left( 2 + q^{16} \right)  +
    4\,q^4\,\left( 1 + q \right) \,\left( 1 + q^2 \right) \,\left( 2 + q^8 \right) \,\lambda  +
    4\,{\left( 1 + q \right) }^2\,\left( 2 + q^4 \right) \,{\lambda }^2 \right).$$
After normalizing by the variance, we see that in this case the fourth moments converge to a rational function of $q$.
\end{example}
\subsection*{Funding}
Funding for this work was provided by the National Science Foundation (\#DMS-0504198).

\subsection*{Acknowledgement} We would like to thank J. Mingo and R. Speicher for helpful
comments on the earlier draft of this paper and for information about their research \cite{mingo-2008} and \cite{Speicher}, to P. \'Sniady for clarifying certain aspects
 of the Brauer product, and to V. Pierce for a discussion that motivated this research and for further
 help with computer-assisted computations.
 We further thank an anonymous referee for a simpler  proof of Lemma \ref{L-NC} and for comments that improved presentation.

\bibliographystyle{acm} 
\bibliography{matrix-models-re,../vita}

\end{document}